\documentclass[11pt,reqno]{amsart}
\usepackage{amsmath,amsthm,amsfonts,amssymb,mathrsfs,bm,graphicx,stmaryrd}

\usepackage{mathtools}
\usepackage{dsfont}
\usepackage{multicol}
\usepackage[colorlinks=true,linkcolor=blue]{hyperref}
\hypersetup{bookmarksdepth=1}
\usepackage{enumitem}
\usepackage{bbm}
\usepackage{mathrsfs}
\usepackage{versions}
\usepackage{subcaption}
\usepackage{tikz}
\usepackage[letterpaper,hmargin=1.0in,vmargin=1.0in]{geometry}
\parskip 	\smallskipamount

\renewcommand{\P}{\mathbb{P}}

\newcommand{\ce}{\mathcal{E}}

\newcommand{\Z}{\mathbb{Z}}
\newcommand{\N}{\mathbb{N}}
\newcommand{\R}{\mathbb{R}}
\newcommand*{\wih}{\widehat}

\newcommand{\vep}{\varepsilon}

\def\l{\left}
\def\r{\right}
\def\<{\langle}
\def\>{\rangle}
\DeclareMathOperator{\sech}{sech}

\newtheorem{theorem}{Theorem}[section]

\newtheorem{lemma}[theorem]{Lemma}
\newtheorem{proposition}[theorem]{Proposition}

\theoremstyle{remark}
\newtheorem{remark}{\bf Remark}

\numberwithin{equation}{section}

\usepackage{palatino}

\title[Right tail estimates for the edge eigenvalues of beta ensembles]{Optimal upper tail estimates for the edge eigenvalues of $\beta$-ensembles}
\author[Baslingker]{Jnaneshwar Baslingker}

\begin{document}
\begin{abstract}
 Hermite and Laguerre $\beta$-ensembles are important and well studied models in random matrix theory, with the special cases $\beta=1,2,4$ corresponding to classical matrix ensembles. Although sharp moderate-deviation estimates have been established for the largest eigenvalues \cite{BBBK25}, substantially less is known for other edge eigenvalues. Even for the limiting distributions $TW_\beta^{(k)}$, the leading constants in the tail asymptotics are not known for general $\beta$ and $k\ge2$. In this paper, for each fixed $k\ge2$, we prove matching upper and lower moderate deviation bounds for the right tail of the $k$-th largest eigenvalue, with the optimal exponential constant $2\beta k/3$, for general $\beta$. We also establish sharp lower tail estimates for the second largest eigenvalue of the Gaussian and Laguerre orthogonal ensembles, with exponential constant $1/24$, same as for the largest eigenvalue. For $TW_\beta^{(k)}$, we obtain optimal upper and lower tail asymptotics for general $\beta$ and $k\geq 1$. We also obtain new stochastic domination results for the edge eigenvalues. Our proofs combine variational arguments from tridiagonal matrix models, last passage percolation, stochastic comparisons of sums of edge eigenvalues, and superposition-decimation identities.
\end{abstract}
\address{Jnaneshwar Baslingker,  Department of Mathematics, University of Toronto, Canada}
\email{j.baslingker@utoronto.ca}
\maketitle

\section{Introduction and statements of main results} 
The study of random matrices with i.i.d. (real or complex) Gaussian entries goes back to almost
a century. The distribution of the eigenvalues of such matrices (GUE/GOE) and their
counterparts for the corresponding Wishart matrices (LUE/LOE) are important and have been extensively
studied. The universality of these models has been established across several landmark papers \cite{BEY14, S99,TV10}. These eigenvalue ensembles are special cases of a more general point process called the $\beta$-ensembles (for a real parameter $\beta>0$), which, starting from the work of Dyson in the 60s, have become an important topic of study over the last decades \cite{RRV, ES07, DE, KJ, BF97, KN04} (see also \cite{AGZ09} for an overview of $\beta$-ensembles). Hermite, Laguerre and Jacobi $\beta$-ensembles are the most well-studied; in this paper we shall only consider the Hermite and Laguerre $\beta$-ensembles. The special cases $\beta=1,2,4$ are exactly solvable and correspond to the eigenvalues of some matrix models (G(O/U/S)E respectively for the Hermite case and L(O/U/S)E respectively for the Laguerre case) and in general are more well studied (see \cite{Dy63, TW94}). 

In this paper, our main topic of interest is the $k$-th largest eigenvalue of Hermite and Laguerre $\beta$-ensembles. It was shown in a seminal paper of Ram\'irez, Rider and Vir\'ag \cite{RRV} (following heuristic approach outlined in \cite{ES07}) that the largest eigenvalue of both the Hermite and Laguerre $\beta$-ensembles (under appropriate normalisation) converge weakly to a distribution known as the Tracy-Widom $\beta$ distribution, which contains the classical (GUE) Tracy-Widom distribution as the $\beta=2$ special case. The tail behaviour of the Tracy-Widom distribution was also determined in \cite{RRV} where they showed that its distribution function $F_{\beta}$ satisfies $1-F_{\beta}(x)=\exp(-\frac{2\beta}{3}(1+o(1))x^{3/2})$ as $x\to \infty$ and $F_{\beta}(x)=\exp(-\frac{\beta}{24}(1+o(1))|x|^3)$ as $x\to -\infty$. 

It is important in many applications, that the pre-limiting (centred and scaled) largest eigenvalues also show similar tail behaviour. Although such estimates existed with the correct exponents $x^{3/2}$ for the upper tail and $x^{3}$ for the lower tail in \cite{LR10} (see also \cite{BGHK21} for the lower bound for lower tail in the Laguerre case) and for the cases $\beta=1,2$ (see e.g.\ \cite{PZ17,L07,BX23, BBF22} and the references therein), tail estimates with the optimal constants in the exponents had been missing until the work in \cite{BBBK25}. In addition, Lambert and Paquette~\cite{LP20},  \cite{LP23} had proved strong embedding results that imply sharp moderate deviation results for the largest eigenvalue, but apparently for a much smaller range of deviations (about $\log n$) than  covered in \cite{BBBK25} (powers of $n$). 

However, in sharp contrast with the largest eigenvalue, substantially less is known for the $k$-th largest eigenvalues where $k\geq 2$. Even the leading constants in the tail asymptotics of the limiting distributions, denoted by $F_{\beta}^{(k)}$, are not known in general. In the integrable case ($\beta=2$), determinantal methods give tail estimates for higher eigenvalues (see \cite[Lemma 2.1]{D24} and \cite{TW94}). For general $\beta$, no analogue with the sharp constant appears to be available. The lower tail is even less understood for fixed $k\ge2$.

At the finite $n$ level, the situation is similarly incomplete. For the largest eigenvalue, moderate-deviation estimates with the correct exponents was proved in \cite{BBBK25}. By contrast, for the $k$-th largest eigenvalue with $k\ge2$, sharp tail estimates (optimal exponential constant) with matching upper and lower bounds, have not previously been established. Although the correct exponents $x^{3/2}$ for the upper tail (upper bounds, $\beta=2$ case) and $x^{3}$ for the lower tail (upper bounds, $\beta>0$) have been established in \cite{DV21}, tail estimates with the optimal constants in the exponents had been missing. 

In this paper (Theorems \ref{t: right tail upper bound for beta geq 2}, \ref{t: rtlb}, \ref{t: ltub}) we establish the optimal moderate deviation estimates, for general Hermite and Laguerre $\beta$-ensembles, in the case of upper tails. While for lower tails, we establish optimal moderate deviation estimates for $\beta=1$ and $k=2$. Surprisingly the lower tail estimates, even though proven for $\beta=1$ and $k=2$, rigorously establishes that the lower tails behaviour does not change with $k$, in contrast with the upper tails.

Tracy-Widom distributions are important even beyond the random matrix theory. In the theory of the Kardar-Parisi-Zhang (KPZ) universality class, the classical Tracy-Widom laws play a role analogous to that of the Gaussian distribution in the theory of sums of independent random variables, arising as universal scaling limits for fluctuations in $(1+1)$-dimensional KPZ models. More generally, the collection of higher edge eigenvalues is related to the multi-line structure underlying KPZ models. The $k$-th largest eigenvalue corresponds, in various exactly solvable models, to higher lines of associated line ensembles or to multi-path statistics. There are also pre-limiting correspondences with distributional identities between multi-path statistics in certain KPZ models and the sum of edge eigenvalues of classical random matrix ensembles. It is therefore natural to expect that sharp moderate-deviation estimates for the $k$-th largest eigenvalue will have applications to the study of higher-line fluctuations and multi-path statistics in the KPZ universality class. 

We now move towards formulating precise statements of our results.

\subsection{Moderate deviations in Gaussian and Laguerre $\beta$-ensembles}

Throughout this subsection $n$ will denote a positive integer. For any real $\beta>0$, consider the probability density function of $\lambda_1\geq \lambda_2\geq\dots\geq\lambda_n\in \mathbb{R}$ given by 
\begin{align}\label{Hermite density}
\mathbb{P}_{n}^{\beta}\l(\lambda_1,\lambda_2,\dots,\lambda_n\r)=\frac{1}{Z_n^{\beta}}\prod_{k=1}^{n}e^{-\frac{\beta}{4}\lambda_k^2}\times\prod\limits_{j<k}\left|\lambda_j-\lambda_k\right| ^{\beta}. 
\end{align}

For $\beta=1,2$ or $4$, this distribution corresponds to the joint density of eigenvalues of $n\times n$ Gaussian orthogonal, unitary or symplectic matrices, G(O/U/S)E of (for general $\beta$ this density is also known as Hermite $\beta$-ensemble) random matrix theory.

For $m\in\mathbb{R}$ with $m>n-1$, consider the following joint density of $\lambda_1\geq\lambda_2\geq \dots\geq\lambda_n\geq 0 $,
\begin{align}\label{Laguerre density}
\mathbb{P}_{n,m}^{\beta}\l(\lambda_1,\lambda_2,\dots,\lambda_n\r)=\frac{1}{Z_{n,m}^{\beta}}\prod_{k=1}^{n}\lambda_k^{\frac{\beta}{2}\l(m-n+1\r)-1}e^{-\frac{\beta}{2}\lambda_k}\times \prod\limits_{j<k}\left|\lambda_j-\lambda_k\right| ^{\beta}.
\end{align}

When $m$ is an integer and $\beta=1,2$ or $4$, this distribution corresponds to the joint density of eigenvalues of  Laguerre orthogonal, unitary or symplectic matrices, L(O/U/S)E (for general $\beta$ the density is known as Laguerre $\beta$-ensemble) of random matrix theory. 

Despite being the focus of research for many years, there were no matrix models for general Hermite $\beta$ and Laguerre $\beta$-ensembles, until the work of 
Dumitriu and Edelman \cite{DE} who gave tridiagonal matrix models for \eqref{Hermite density} and \eqref{Laguerre density}. Define 
\begin{align}\label{eq: tridiag matrix}
 H_{n,\beta}:=\frac{1}{\sqrt{\beta}}\begin{bmatrix}
X_1 & Y_1 & 0 &\cdots 0\\
Y_1 & X_2 & Y_2 & \cdots 0\\
0 & Y_2 & X_3 &\cdots 0\\
\vdots & \vdots & \vdots &\vdots
\end{bmatrix}_{n\times n}
    \hspace{1cm}
B_{n,m,\beta}:=\frac{1}{\sqrt{\beta}}\begin{bmatrix}
D_1 & 0 & 0 &\cdots 0\\
C_1 & D_2 & 0 & \cdots 0\\
0 & C_2 & D_3 &\cdots 0\\
\vdots & \vdots & \vdots &\vdots
\end{bmatrix}_{n\times n}.
\end{align}
$H_{n,\beta}$ 
is the symmetric tridiagonal matrix with $X_i\sim N\l(0,2\r)$ and $Y_i\sim \chi_{\beta\l(n-i\r)}$ (here $Y\sim \chi_{\nu}$ means $Y>0$ and $Y^2\sim \chi_{\nu}^2$) and the entries are independent up to symmetry.
 $B_{n,m, \beta}$ is the bidiagonal matrix  with $D_i\sim \chi_{\beta\l(m-i+1\r)}$ and $C_i\sim \chi_{\beta\l(n-i\r)}$, where $m\in \mathbb{R}$ and $m+1>n$ and all the entries are independent.
Dumitriu and Edelman showed that the $n$ eigenvalues of $H_{n,\beta}$ have the joint law given by \eqref{Hermite density}
 and the eigenvalues of $L_{n,m,\beta}=B^t_{n,m, \beta}B_{n,m, \beta}$  
 have the joint density given by \eqref{Laguerre density}.
 
 We denote $k$-th largest eigenvalue of $H_{n,\beta}$ and $L_{n,m,\beta}$ by $\lambda_k$. We also use the notation \newline $\lambda_k^{\l(n,H,\beta\r)},\lambda_k^{\l(n,m,L,\beta\r)}$ for 
 the $k$-th largest point in \eqref{Hermite density} and \eqref{Laguerre density} respectively, then
 \begin{align}\label{eq: Dimtriu Edelman result}
     \lambda_{k}\l(H_{n,\beta}\r)\overset{d}{=}\lambda_k^{\l(n,H,\beta\r)}
     \quad   \text{and}\quad\lambda_{k}\l(L_{n,m, \beta}\r)\overset{d}{=}\lambda_k^{\l(n,m,L,\beta\r)}.
 \end{align}

Ram\'irez-Rider-Vir\'ag \cite{RRV} showed that the largest eigenvalues of the above random tridiagonal matrices, after suitable shifting and scaling, converge to the point process of eigenvalues of a second order stochastic differential operator called the stochastic Airy operator ($\mbox{SAO}_{\beta}$). In particular, using \eqref{eq: Dimtriu Edelman result}, their results show that
\begin{align}\label{eq: Hermite TW convergence}
\l(\frac{\lambda_k^{\l(n,H,\beta\r)}}{\sqrt{ n}}-2\r)n^{2/3}&\implies \mbox{TW}_{\beta}^{(k)},
    \\ \label{eq: Laguerre TW convergence}
    {\l(\sqrt{m n}\r)^{1/3}}{\l(\sqrt{m}+\sqrt{n}\r)^{2/3}}&\l(\frac{\lambda_k^{\l(n, m, L,\beta\r)}}{\l(\sqrt{m}+\sqrt{n}\r)^{2}}-1\r)\implies \mbox{TW}_{\beta}^{(k)},
\end{align}
where $\mbox{TW}_\beta^{(k)}$ is distributed as the negative of the $k$-th smallest eigenvalue of $\mbox{SAO}_\beta$. For $k=1$, we denote $TW_{\beta}^{(k)}$ as $TW_{\beta}$. \cite{RRV} also found various properties of  $\mbox{TW}_{\beta}$ (read as Tracy-Widom~$\beta$). Of particular interest to us is the tail behaviour which they found. As $a\rightarrow\infty$,
\begin{align}\label{eq: TW right tail bounds}
    \P\l(TW_{\beta}>a\r)&=\exp\l(-\frac{2}{3}\beta a^{3/2}\l(1+o(1)\r)\r),\\\label{eq: TW left tail bounds}
        \P\l(TW_{\beta}<-a\r)&=\exp\l(-\frac{1}{24}\beta a^{3}\l(1+o(1)\r)\r).
\end{align}
Prior to their work, Baik-Buckingham-DiFranco~\cite{BBD08} had already found much more  precise asymptotics for $\beta=1,2,4$. In \cite{DV13}, Dumaz and Vir\'ag sharpened the  right tail estimates for general $\beta$ by finding lower order corrections. Although we require new insights in using the tridiagonal matrix structure of $\beta$-ensembles to prove both upper and lower bounds for the upper tails, our methods are similar to those in \cite{BBBK25}. 

The $TW_{\beta}^{(k)}$ distributions (and their pre-limiting variants) also appear in many contexts outside random matrices. For instance, these laws describe the fluctuations of multi-path last passage times in percolation \cite{KJ}, longest $k$-disjoint increasing subsequence sums of a random permutation \cite{BDJ98}. It is therefore natural, and important in many applications, to investigate the tail bounds for $ \lambda_k^{\l(n,H,\beta\r)}$ and $ \lambda_k^{\l(n,m,L,\beta\r)}$. 

Even the analogues of \eqref{eq: TW right tail bounds}, \eqref{eq: TW left tail bounds} are not known for $k\geq 2$ and general $\beta$, let alone moderate deviation results for the prelimiting quantities in \eqref{eq: Hermite TW convergence}, \eqref{eq: Laguerre TW convergence}. Other than obtaining the moderate deviation results for the upper tails (see Theorem \ref{t: right tail upper bound for beta geq 2}-\ref{t: rtlb}), we also obtain analogues of \eqref{eq: TW right tail bounds} (for $\beta\geq 2/k$), \eqref{eq: TW left tail bounds} (for $\beta>0 $ and $k\geq 2$). See Theorem \ref{thm: TW tails}. 

We summarize here some of the known results about the finite size tail estimates for statistics of interest that converge to the Tracy-Widom distribution. It was considered in \cite{L04,L07} that the right tail inequality for GUE (also LUE) may be shown to follow from results of Johansson \cite{KJ} which uses sub-additivity arguments and large deviation asymptotics. The optimal right tail bounds in the case of GUE, up to the constants in the exponential tails, were obtained in \cite{PZ17}. For GOE, \cite{BX23} gives the right tail upper bound with optimal constants in a certain parameter regime. The right tail upper bounds, with the optimal constant for last passage model that corresponds to LOE was obtained in \cite{BBF22}. By the superposition-decimation procedure \cite{FR01}, similar bounds for GOE can be also be found (see \cite{Ledoux09}). Precise upper tail estimates are also available for some related positive temperature models in the KPZ universality class, see e.g. \cite{DG23, GH22}.

Even for the exactly solvable cases, the left tail bounds are generally less accessible, often requiring more difficult Riemann-Hilbert analysis. The left tail inequalities for last passage times in planar last passage percolation with geometric weights, which corresponds to largest element in the Meixner ensemble, are proven in \cite{BDMMZ01} using Riemann-Hilbert methods. In case of passage times in Poissonian LPP, which is obtained by a suitable limit of largest element in Meixner ensemble, optimal estimate for upper tail and lower tail was shown in \cite{LM01,LMS02} using Riemann-Hilbert methods. Coming to the case of Hermite and Laguerre ensembles with general $\beta$, Ledoux and Rider \cite{LR10} provide the tail bounds with the correct exponents $3/2$ in the upper tail and $3$ in the lower tail for general $\beta$ (but not the optimal constant $2\beta/3$ and $\beta/24$). Upper bounds for the right and left tails for Hermite and Laguerre ensembles were obtained in \cite[Theorem 1,2]{LR10}. But it was remarked that the lower bounds were obtained only in the Hermite ensemble. The left tail lower bound for Laguerre ensemble (general $\beta$) was obtained in \cite[Theorem 2]{BGHK21}. In \cite{GH23} the 3/2 and 3 exponents (without the optimal constants) for the right tail and left tails was conditionally obtained in last passage percolation models via bootstrapping and geodesic geometry. 

The problem of obtaining sharp moderate deviation results, optimal up to the constants in the exponents, for both Hermite and Laguerre ensembles, for general $\beta$, was finally resolved in \cite{BBBK25}. Similar results for height functions of stochastic six vertex models were obtained in \cite{GS25}. For $k\geq 2$, very little is known about the tail estimates of $k$-th largest eigenvalue fluctuations of Hermite and Laguerre $\beta$-ensembles. For $\beta=2$, tail estimates of $TW_{\beta}^{(k)}$ are known using the determinantal structure \cite{D24, TW94}. The lower tail estimates are also known, only for $\beta=2$ (\cite{TW94}). For the prelimiting objects, the correct exponents have been found (for special values of $\beta$) in \cite{DV21}.

We state the main theorems now. Theorems~\ref{t: right tail upper bound for beta geq 2}-\ref{t: ltub} essentially assert that the pre-limiting random variables in \eqref{eq: Hermite TW convergence} and \eqref{eq: Laguerre TW convergence} have roughly the same upper tail probabilities as product of $k$ independent Tracy-Widom distributions. On the lower tail, for $\beta=1$, both the first and the second largest eigenvalue have the same tails.

It is widely believed that the left tails should not depend on $k$, for all $\beta>0$. The heuristic is that to push the $k$-th largest eigenvalue too much to the left, all the lower eigenvalues need to be pushed, which makes the cost comparable to that of pushing the largest eigenvalue too much to the left. Hence the left tails are expected not to depend on $k$. Indeed for $\beta=2$, this is known in the tails of $TW_{2}^{(k)}$ \cite{TW94}. We prove that this heuristic is correct for $\beta=1$ and $k=2$ in the pre-limiting models. To the best of our knowledge, this is the first result establishing the heuristic. Also, to the best of our knowledge, our results provide the first optimal moderate-deviation estimates for the upper tail, for $k\ge2$, in any model in the KPZ universality class.

\begin{theorem}
\label{t: right tail upper bound for beta geq 2}
For the right tail upper bound we have
    \begin{enumerate}[label=(\roman*), font=\normalfont] 
        \item For any $\beta\geq 2$ or $\beta=1$. Let $\varepsilon>0$ and $k\in\N$. We can choose $n_{\varepsilon,\beta,k} \in \N$ and $\gamma_{\varepsilon,\beta,k},t_{\vep,\beta,k}>0$ such that for all $n\geq n_{\varepsilon,\beta,k}$ and for all $t$ with $t_{\vep,\beta,k} \leq t\leq \gamma_{\vep,\beta,k}n^{2/3}$,  the following holds 
\begin{align*}
    \P \l( \lambda_k^{\l(n,H,\beta\r)} \geq 2 \sqrt{n}+tn^{-1/6}\r) \leq   \exp\l({-\frac{2\beta k}{3}t^{3/2}\l(1-\varepsilon\r)}\r).
\end{align*}

\item Fix any $\beta \geq 2$ or $\beta=1$. Let $ \varepsilon>0,\ c\geq 1$ and $k\in\N$ and $m=cn$. Then there exist $n_{\varepsilon,\beta,c,k} \in \N$ and $\gamma_{\varepsilon,\beta,c,k},t_{\vep,\beta,c,k}>0$ such that the following holds. For any real number $t$ satisfying $t_{\vep,\beta,c,k}\leq t \leq \gamma_{\varepsilon,\beta,c,k}n^{2/3}$ and $n\geq n_{\varepsilon,\beta,c,k}$, we have
        \[
        \P \l(\lambda_{k}^{(n,m,L, \beta)} \geq \l(\sqrt{m}+\sqrt{n} \r)^2+t \l(\sqrt{m n} \r)^{-1/3}\l(\sqrt{m}+\sqrt{n} \r)^{4/3} \r) \leq \exp\l({-\frac{2\beta k}{3}t^{3/2}\l(1-\varepsilon\r)}\r).
        \]
    \end{enumerate}
\end{theorem}

\begin{remark}
       
   The results of \cite{BBBK25}[Theorem 1.1(iii)] for the largest eigenvalue were less than satisfactory. For Laguerre ensembles, we were unable to obtain the right tail upper bound for general $(m,n)$ in the range $1 \leq \beta <2.$ See \cite{BBBK25}[Remark 1]. The technique used in the proof of Theorem \ref{t: right tail upper bound for beta geq 2} covers this gap, especially for $\beta=1$ which corresponds to classical matrix ensembles. We state the following result, whose proof technique is discussed in Remark \ref{rem: stronger upper tail}.
    \end{remark}
   \begin{theorem}\label{thm: stronger upper tail}
     Fix any $\beta \geq 1$ and $ \varepsilon>0,c\geq 1$ and $m=cn$. Then there exist $n_{\varepsilon,\beta,c} \in \N$ and $\gamma(\varepsilon,c),t_{\vep,\beta,c}>0$ such that the following holds. For any real number $t$ satisfying $t_{\vep,\beta,c}\leq t \leq \gamma(\varepsilon,c)n^{2/3}$ and $n\geq n_{\vep,\beta,c}$, we have
        \[
        \P \l(\lambda_{1}^{(n,m,L, \beta)} \geq \l(\sqrt{m}+\sqrt{n} \r)^2+t \l(\sqrt{m n} \r)^{-1/3}\l(\sqrt{m}+\sqrt{n} \r)^{4/3} \r) \leq \exp\l({-\frac{2\beta}{3}t^{3/2}\l(1-\varepsilon\r)}\r).
        \]
   \end{theorem}

\begin{theorem}
\label{t: rtlb}
For the right tail lower bound we have
    \begin{enumerate}[label=(\roman*), font=\normalfont]
        \item For any $\beta>0$ and any small $\varepsilon>0$ and any $k\in\N$, there exist $n_{\vep,\beta,k} \in \N$ and $\gamma_{\vep,\beta,k},t_{\vep,\beta,k}>0$ such that for all $n \geq n_{\vep,\beta,k}$ and $ t_{\vep,\beta,k} \leq t \leq \gamma_{\vep,\beta,k}n^{2/3}$ we have that 
        \[
        \P \l(\lambda_k^{(n,H,\beta)} \geq 2\sqrt{n}+tn^{-1/6} \r) \geq \exp\l({-\frac{2\beta k}{3}t^{3/2}\l(1+\varepsilon\r)}\r).
        \]
    
        \item Fix any $\beta >0$ and $k\in\N$ and $ \varepsilon,M>0$. Then there exist $n_{\varepsilon,\beta,M,k} \in \N$ and $\gamma_{\vep,\beta,M,k}, t_{\vep,\beta,M,k}>0$ such that the following holds. Consider any pair of natural numbers $\l(m,n\r)$ with $ m+1>n\geq n_{\varepsilon,\beta,M,k}$ and $ \frac{m}{n} \leq M$. Further, consider any real number $t$ satisfying $t_{\vep,\beta,M,k}\leq t \leq \gamma_{\varepsilon,\beta,M,k}n^{2/3}$. Then we have
        \[
        \P \l(\lambda_k^{(n,m,L,\beta)} \geq \l(\sqrt{m}+\sqrt{n} \r)^2+t\l(\sqrt{m n}\r)^{-1/3}\l(\sqrt{m}+\sqrt{n}\r)^{4/3} \r) \geq \exp \l({-\frac{2\beta k}{3}t^{3/2}\l(1+\varepsilon\r)}\r).
        \]
    \end{enumerate}
\end{theorem} 

\begin{theorem}
    \label{t: ltub} 
     For $k\in\{1,2\}$ and $M\geq 1$ and any $\vep,\delta>0$ small enough there exist $n_{\vep,\delta},n_{\vep,\delta,M} \in \N$ and $t_{\vep,\delta},t_{\vep,\delta,M}>0$, such that 
    \begin{enumerate}[label=(\roman*), font=\normalfont]
        \item  for any $n \geq n_{\vep,\delta}$ and $t_{\vep,\delta} \leq t \leq n^{\frac 16-\delta}$ we have
        \[
        \P \l(\lambda_{k}(H_{n,1}) \leq 2\sqrt{n} -tn^{-1/6}\r) \leq \exp \l({-\frac{1}{24}t^{3}\l(1-\varepsilon\r)}\r).
        \]
        \item Consider any pair of natural numbers $\l(m,n\r)$ with $  m+1>n  \geq n_{\varepsilon,\delta,M}$ and $ \frac{m}{n} \leq M$.  Further, consider any real number $t$ with $t_{\vep,\delta,M} \leq t \leq n^{\frac 16-\delta}$. Then we have
        \[
        \P \l(\lambda_{k}(L_{n,m,1}) \leq \l(\sqrt{m} +\sqrt{n}\r)^2-t\l(\sqrt{m n} \r)^{-1/3} \l(\sqrt{m} +\sqrt{n}\r)^{4/3}\r) \leq \exp \l({-\frac{1}{24}t^{3}\l(1-\varepsilon\r)}\r).
        \]
    \end{enumerate}
\end{theorem}
As corollaries of our above results, we obtain the tail behaviour of $TW_{\beta}^{(k)}$.
\begin{theorem}\label{thm: TW tails}
\hspace{1cm}
\begin{enumerate}
     \item  For $k\geq 2$ and $\beta\geq 2/k$, as $a\rightarrow\infty$, we have
    \begin{align*}
\P\l(TW_{\beta}^{(k)}>a\r)\leq \exp\l(-\frac{2\beta k}{3}a^{3/2}(1+o(1))\r).
    \end{align*}
    \item For $k\in\N$ and $\beta>0$, as $a\rightarrow\infty$, we have
    \begin{align*}
\P\l(TW_{\beta}^{(k)}>a\r)\geq \exp\l(-\frac{2\beta k}{3}a^{3/2}(1+o(1))\r).
    \end{align*}
    \item For $k\geq 1$ and $\beta>0$, as $a\rightarrow\infty$, we have
    \begin{align*}
\P\l(TW_{\beta}^{(k)}<-a\r)\leq \exp\l(-\frac{\beta }{24}a^{3}(1+o(1))\r).
\end{align*}
\end{enumerate}
\end{theorem}

\begin{remark}Note that the range of $\beta$ in Theorem \ref{t: right tail upper bound for beta geq 2} is $\{1\}\cup [2, \infty),$ while the range of $\beta$ in Theorem \ref{t: rtlb} is $(0, \infty).$ While we expect that the range of $\beta$ to be $(0,\infty)$ in Theorem \ref{t: right tail upper bound for beta geq 2} too, our proof starts with the $\beta=2$ case and uses the stochastic domination to handle  the values $\beta > 2.$ But we cannot use the stochastic domination technique in the range $\beta \in (0,2)$ as we do not have a base case to start with. The value of $\beta=1$ is handled using the decimation-superposition results.
\end{remark}

The upper bound for the upper tail in the case of the largest eigenvalue in \cite{BBBK25} was first proven for $\beta=2$ case, using the last passage percolation arguments and then the same was extended to all $\beta>2$ using the stochastic comparison of the largest eigenvalues across different $\beta$. Such stochastic comparison result was obtained in \cite{JB23}, which we state below. Denote the left hand side terms of \eqref{eq: Hermite TW convergence},\eqref{eq: Laguerre TW convergence} which are normalized and converge, as $n\rightarrow\infty$ to $TW_{\beta}$, by $N(H_{n,\beta})$ and $N(L_{n,m,\beta})$.
\begin{theorem}[Theorem 1 of \cite{JB23}]\label{Stoch dom}
Let $\frac{\beta_2}{\beta_1}=\frac{m}{n}\in \mathbb{Q}$ and $\beta_2>\beta_1$. Let $s=2/3$. Then
\begin{align*}
    \beta_1^s  N(H_{m,\beta_1})&\ \overset{s.t.}{\geq}  \beta_2^s  N(H_{n,\beta_2}).  \\
     \beta_1^s\  N(L_{m,m\kappa/n,\beta_1})&\ \overset{s.t.}{\geq}  \beta_2^s\  N(L_{n,\kappa,\beta_2}).
\end{align*}
\end{theorem}

As observed in \cite{JB23}, the same argument does not extend to $k$-th largest eigenvalue. But we prove in the following result that one can stochastically compare the sum of top $k$ eigenvalues across different $\beta$. 

\begin{proposition}\label{prop: sum of top k stoch dom}
    Let $\frac{\beta_2}{\beta_1}=\frac{s}{n}\in\mathbb{Q}$ and $\beta_1<\beta_2$. Then for any $k\leq  n\in\N$ and $s\in\N$, we have
    \begin{align*}
\sqrt{\beta_2}\sum\limits_{i=1}^k\lambda_i(H_{n,\beta_2})&\overset{\text{s.t.}}{\leq} \sqrt{\beta_1}\sum\limits_{i=1}^k\lambda_i(H_{s,\beta_1})\\
{\beta_2}\sum\limits_{i=1}^k\lambda_i(L_{n,m,\beta_2})&\overset{\text{s.t.}}{\leq} {\beta_1}\sum\limits_{i=1}^k\lambda_i(L_{s,m\frac{s}{n},\beta_1})
    \end{align*}
\end{proposition}

Theorem \ref{Stoch dom} is finite $n$ analogue of the following stochastic comparison result of \cite{VP2022}, which we will be using in the proof of Theorem \ref{thm: TW tails}. Although originally stated for $k=1$, the same proof works for any $k\in\N$.

\begin{theorem}[Theorem $1$ of \cite{VP2022}]\label{thm:VP result}
    Let $\beta'>\beta>0$ and $\alpha>0$ and $k\in\N$. Then $TW_{\beta}^{(k)}\overset{\text{s.t.}}{\geq}\alpha TW_{\beta'}^{(k)} $ if $\l(\frac{\beta'}{\beta}\r)^{1/3}\leq \alpha\leq \l(\frac{\beta'}{\beta}\r)^{2/3}$.
\end{theorem}

A equivalent form of Theorem \ref{thm:VP result} is that for $s\in\{\frac{1}{3},\frac{2}{3}\}$ and $\beta_2>\beta_1$ and $k\in\N$, we have
\begin{align}\label{eq: TW eq form}
    \beta_{1}^{s}TW_{\beta_1}^{(k)}\overset{\text{s.t.}}{\geq}\beta_{2}^{s}TW_{\beta_2}^{(k)}.
\end{align}

Although Theorem \ref{thm:VP result} holds for both $s\in\{\frac{1}{3},\frac{2}{3}\}$, its finite $n$ analogue is known only for $s=2/3$. There are no known cases of finite $n$ analogues of Theorem \ref{thm:VP result} for $s=1/3$. Unlike for $s=2/3$, even in the paper \cite{VP2022}, there is no explanation for $s=1/3$ stochastic comparison result that emerges from the analysis of stochastic Airy operator. We now provide a finite $n$ version of Theorem \ref{thm:VP result} for $s=1/3$.
\begin{theorem}\label{Stoch dom s=1/3}
Let $r,k\in\Z^+$ and $\beta_2=2(r+1)$ and $\beta_1=2/(r+1)$. Let $s=1/3$ and $n_r=(r+1)n$ and $m_r=(r+1)m$. Then
\begin{align*}
    \beta_1^s  \l(\frac{\lambda_k(H_{n_r+r,\beta_1})}{\sqrt{n_r}}-2\r)\l(n_r\r)^{2/3} \ &\overset{\text{s.t.}}{\geq}\  \beta_2^s  \l(\frac{\lambda_k(H_{n,\beta_2})}{\sqrt{n}}-2\r)n^{2/3},\\
 \beta_1^s  {\l(\sqrt{m_r n_r}\r)^{1/3}}{\l(\sqrt{m_r}+\sqrt{n_r}\r)^{2/3}}&\l(\frac{\lambda_k(L_{n_r+r, m_r+r, L,\beta_1})}{\l(\sqrt{m_r}+\sqrt{n_r}\r)^{2}}-1\r) \ \\ \overset{\text{s.t.}}{\geq}\  &\beta_2^s  {\l(\sqrt{m n}\r)^{1/3}}{\l(\sqrt{m}+\sqrt{n}\r)^{2/3}}\l(\frac{\lambda_k(L_{n,m,\beta_2})}{\l(\sqrt{m}+\sqrt{n}\r)^{2}}-1\r).   
\end{align*}
In particular, in both cases, taking $n\rightarrow\infty$ limit, we obtain \eqref{eq: TW eq form} for the above mentioned values.
\end{theorem}

Although we do not have a proof, we believe that Proposition \ref{prop: sum of top k stoch dom} holds for $k$-th largest eigenvalues as well. We now introduce the superposition and decimation results which we will use in the proof of Theorem\ref{t: ltub}, Theorem \ref{thm: stronger upper tail} and Theorem \ref{thm: TW tails}.


\subsection{Superposition and decimation results:}

Superposition of matrix ensembles refers to the ensemble constructed from union of two independent ensembles, while a decimation refers to the new ensemble obtained by observing only every second eigenvalue. In the cases of the classical matrix ensembles with
orthogonal symmetry, it is known that forming superpositions and decimations gives classical matrix ensembles with unitary and symplectic symmetry. See \cite{F24, FR01, FR05, FR04, F07} for more detailed expositions and results.

 To state these consequences, let $\mathrm{ME}_{\beta,n}(g)$ refer to the matrix ensemble specified by the density function,

\begin{align}
\label{eq: decimation density}
\frac{1}{C} \prod_{l=1}^n g(x_l) \prod_{1 \le j < k \le n} |x_k - x_j|^\beta 
\end{align}

(unless stated otherwise, $C$ will denote some normalization), and let the eigenvalues in \eqref{eq: decimation density} be ordered
\begin{align*}
    x_1 > x_2 > \dots > x_n.
\end{align*}

Denote by $\mathrm{D}_p(\mathrm{ME}_{\beta,n}(g))$ the joint marginal distribution of $x_p, x_{2p}, x_{3p}, \ldots$ in (1.12). Following the nomenclature of \cite{F07}, the D here stands for a decimation procedure. Recall that for $\beta=1,2,4$ in \eqref{Hermite density}, we refer it as $GOE_n/GUE_n/GSE_n$, respectively. For \eqref{Laguerre density}, we refer the same as $LOE_{m,n}/LUE_{m,n}/LSE_{m,n}$ respectively. More generally, \cite{F07}[Theorem 2] gives,
\begin{align}\label{eq: gen dec results}
\mathrm{D}_{r+1}(\mathrm{ME}_{2/(r+1),(r+1)n+r}(x^a e^{-x})) &= \mathrm{ME}_{2(r+1),n}(x^{(r+1)a+2r} e^{-(r+1)x}), \notag \\
\mathrm{D}_{r+1}(\mathrm{ME}_{2/(r+1),(r+1)n}(e^{-x})) &= \mathrm{ME}_{2(r+1),n}(e^{-(r+1)x}), \notag \\
\mathrm{D}_{r+1}(\mathrm{ME}_{2/(r+1),(r+1)n+r}(e^{-x^2})) &= \mathrm{ME}_{2(r+1),n}(e^{-(r+1)x^2}). 
\end{align}

For the following relations, let $n_r=(r+1)n+r$. Then \eqref{eq: gen dec results} imply that,
\begin{align}\label{eq:GOE-GSE relation}
D_2(GOE_{2N+1})\overset{d}{=} \sqrt{2} GSE_N\\
D_{r+1}(H_{(r+1)n+r,2/(r+1)})\overset{d}{=} \sqrt{r+1}H_{n,2(r+1)}\label{eq: gen hermite comp}\\
D_{r+1}(L_{n_r,m_r,2/(r+1)})\overset{d}{=} ({r+1})L_{n,m,2(r+1)}\label{eq: gen laguerre comp}.
\end{align}

By \cite{FR01}[Theorem 5.2], we also have
\begin{align}\label{eq: GOE-GUE relation}
    D_2(GOE_n\cup GOE_{n+1})&\overset{d}{=}GUE_{n}\\
D_2(LOE_{2m+1,2n+1})&\overset{d}{=}2LSE_{m,n}\\
D_2(LOE_{m,n}\cup LOE_{m+1,n+1})&\overset{d}{=}LUE_{m,n}\label{eq:LOE-LSE relation}.
\end{align}

\subsection{Outline of the proofs.} We shall now describe briefly the ideas that go into the proofs of the main results. We will start with the tail estimates for $k$-th largest eigenvalues in \eqref{eq: Hermite TW convergence}, \eqref{eq: Laguerre TW convergence}. Both the upper bounds and lower bounds of upper tails require new insights in to exploiting the tridiagonal matrix structure. For the right tails, lower bound can be established by constructing suitable test functions. For the lower bound, we modify $H_{n,\beta}$ and $L_{n,m,\beta}$ by replacing $\chi_{\beta(n-i)}$ random variables with $\sqrt{\beta\l(n-i\r)}+\frac{\zeta_i}{\sqrt{2}}$ for all $i\leq p$, where $\zeta_i$ are i.i.d.\ $N\l(0,1\r)$ random variables. We denote these modified matrices as $\wih{H}_{n,\beta,p}$ and $\wih{L}_{n,m,\beta,p}$. We use KMT theorem to couple $\chi_k$ random variables in tridiagonal matrices with Gaussians and then use Gershgorin circle theorem along with Weyl's inequality to control the difference in the spectrums of the modified and the original tridiagonal matrices. This is Proposition \ref{eq: estimate of off diagonals of the difference} proved in the beginning of Section \ref{s: rtlb}. Below we outline different pieces of the arguments for the Hermite ensemble in more detail. Similar arguments hold for Laguerre ensemble.

\subsubsection{\textbf{Upper bounds for the right tails.}} We first obtain sharp right tail upper bounds for $\beta=2$. This is done using the connection of the eigenvalues of Hermite ensemble for $\beta=2$ and the Brownian last passage percolation. As the last passage percolation model satisfies super-additivity, we use large deviation theory to obtain sharp upper bound for the right tails. Then one can obtain the same for the sum of top $k$ eigenvalues for any $\beta>2$,
using the stochastic domination result from Proposition \ref{prop: sum of top k stoch dom}.
This results in sharp upper bounds for the right tails for $\beta\geq 2$, for the range of $t$ up to $\delta n^{2/3}$. To cover $\beta=1$ case, we use the superposition decimation results \eqref{eq: GOE-GUE relation}, to compare $\beta=2$ and $\beta=1$ cases. The proofs are in Section \ref{sec:rtub}.

\subsubsection{\textbf{Lower bounds for the right tails.}} We modify the matrix $H_{n,\beta}$ to $\wih{H}_{n,\beta,p_k}$, by replacing  $p_k\asymp  n^{1/3}$ of the $\chi$ random variables with  appropriate Gaussian random variables. Then we  construct vectors $v_k$ and use Gaussian tail bounds to conclude that  
\begin{align*}
    \< \wih{H}_{n,\beta,p_k}v_i,v_i\>\geq\l(2\sqrt{n}+tn^{-1/6}\r) \<v_i,v_i\>
\end{align*}
occurs with probability at least as that of the lower bound of Theorem \ref{t: rtlb}. The vector $v_1$ we construct is such that the non-zero part of the vector is simply the discretized form of the function used in \cite{RRV} to obtain lower bound for right tails of $TW_{\beta}$. To construct the remaining vectors we shift the support of $v_1$ and analyze the resulting changes in the calculations done for $v_1$. This method of constructing new vectors with shifted support from $v_1$, is used in \cite{DV21} to get the exponents $x^3$ for the left tails of $k$-th largest eigenvalues of $H_{n,\beta}$. We analyze the effect of shifting the support and show that the precise constant can be obtained for the right tail lower bound. We also show exactly why such analysis fails to obtain sharp constants for upper bounds for the left tails. We then use the KMT theorem (in Proposition \ref{eq: estimate of off diagonals of the difference}) to argue that the lower bound carries over to the right tails of  $\lambda_{k}(H_{n,\beta})$. The proofs are  in Section \ref{s: rtlb}.

\subsection*{Organization of this paper} The rest of the paper is organized as follows. We prove the right tail upper bounds (Theorem \ref{t: right tail upper bound for beta geq 2}) in Section \ref{sec:rtub}, right tail lower bounds (Theorem \ref{t: rtlb}) in Section \ref{s: rtlb}. In each section, the Hermite case is dealt with first, and then the more complicated Laguerre case is tackled. The proof sketch of Theorem \ref{thm: stronger upper tail}, which is similar to the proof of Theorem \ref{t: right tail upper bound for beta geq 2} is discussed in Remark \ref{rem: stronger upper tail}. Theorem \ref{t: ltub} is proven in Section \ref{sec: left tail bounds}. The results on stochastic domination such as Proposition \ref{prop: sum of top k stoch dom}, Theorem \ref{Stoch dom s=1/3}, Theorem \ref{thm: TW tails} are proven in Section \ref{sec: stoch dom}.

 \noindent {\bf Acknowledgments.} The author thanks Duncan Dauvergne and B\'alint Vir\'ag for several insightful discussions throughout the project.
\section{Lower bounds for right tails} 
\label{s: rtlb}
For the remaining proofs of the tail bounds we use the below proposition which allows us to study modifications of $H_{n,\beta}$ and $B_{n,m,\beta}$. For $p \leq n$, let $\wih{H}_{n,\beta,p}$ be the matrix where $\chi_{\beta\l(n-i\r)}$ in $H_{n,\beta}$ are replaced by $\xi_i:=\sqrt{\beta\l(n-i\r)}+\frac{\zeta_i}{\sqrt{2}}$ for all $i\leq p$, where $\zeta_i$ are i.i.d.\ $N\l(0,1\r)$ random variables. That is,

\begin{align*}
\wih{H}_{n,\beta}=\frac{1}{\sqrt{\beta}}\begin{bmatrix}
    X_{1} & \xi_1 & 0 & 0 & \cdots & 0 & 0 & \cdots & \cdots & 0 \\
    \xi_1 & X_{2} & \xi_2 & 0 & \cdots & 0 & 0 & \cdots & \cdots & 0\\
    0 & \xi_2 & X_{3} & \xi_3 & \cdots & 0 & 0 & \cdots & \cdots & 0\\
    \vdots & \ddots & \ddots & \ddots & \ddots & \vdots & \vdots& \vdots&  \vdots & \vdots \\
    0 & \cdots  & 0 & \xi_{p} & X_{p} & Y_{p+1}& 0 &  \cdots & \cdots & 0 \\
    0 & \cdots & \cdots & 0 & Y_{p+1} & X_{p+1} & Y_{p+2} & 0 & \cdots &  0\\
    0 & \cdots & \cdots & \cdots & 0 & Y_{p+2} & X_{p+3} & Y_{p+3} & \cdots  & 0\\
    \vdots & \vdots & \vdots & \vdots & \vdots & \vdots & \ddots& \ddots& \ddots\\
     \vdots & \vdots & \vdots & \vdots & \vdots & \vdots & \ddots& Y_{n-2} & X_{n-1} & Y_{n-1}\\
    0 & \cdots & 0 & 0 & \cdots &  \cdots &  \cdots &  \cdots & Y_{n-1} &   X_{n}\\
\end{bmatrix}.
\end{align*}
where $X_1, X_2 \dots , X_n$ and $ Y_{p+1}, \dots, Y_{n-1}$ are same as in \eqref{eq: tridiag matrix}. Let $\wih{B}_{n,m,\beta,p}$ be the matrix where $D_{i}$s and $C_i$s in $B_{n,m,\beta}$ are replaced by $\sqrt{\beta\l(m+1-i\r)}+\frac{\zeta_i}{\sqrt{2}}$ and $\sqrt{\beta\l(n-i\r)}+\frac{\zeta'_i}{\sqrt{2}}$ respectively for all $i\leq p$. Here $\zeta_i$s and $\zeta'_i$s are independent families of i.i.d.\ $N(0,1)$ random variables. The entries of $\wih{H}_{n,\beta,p}$ and $\wih{B}_{n,m,\beta,p}$ are independent up to symmetry. We define $\wih{L}_{n,m,\beta,p}=\wih{B}_{n,m,\beta,p}^t\wih{B}_{n,m,\beta,p}$.

In what follows, $\|\cdot \|_2$ of a matrix denotes the spectral norm. We use the following result which is proved in \cite{BBBK25}[Proposition 3.1]. This result allows us to replace the $\chi$-distributed random variables by suitably matched Gaussian random variables, at the cost of an error that is negligible for the probability estimates under consideration.
\begin{proposition}[Proposition 3.1 of \cite{BBBK25}]\label{eq: estimate of off diagonals of the difference} For all $p\leq n/2$ and for all $x>0$ we have the following:
    $\left\lVert \wih{H}_{n,\beta,p}-H_{n, \beta}\right\rVert_2\leq c_{\beta} \frac{\log n+2x}{ \sqrt{n}}$ with probability at least $1-C n\l(3e^{-x}+2e^{-\frac{\beta\l(n-p\r)}{2}}\r)$ (for some constants $C,c_{\beta}>0$ and $n$ sufficiently large). We also have for $\frac mn$ uniformly bounded by some constant, $\l\lVert\wih{L}_{n,m,\beta,p}-L_{n,m,\beta}\r\rVert_2\leq c_{\beta} \l({\log n+2x}\r)$ with probability at least $1-C n\l(18e^{-x}+8e^{-\frac{\beta\l(n-p\r)}{2}}\r)$ (for some constants $C,c_{\beta}>0$ and $n$ sufficiently large.
\end{proposition}

The following lemma allows to approximate infinite integrals with infinite Riemann sums.
\begin{lemma}[Lemma 3.2 of \cite{BBBK25}]
\label{lem:aprox_inf_int}
Let $f$ be a continuously differentiable function defined on $\R$ such that both $f$ and $f'$ are integrable. Fix $h>0$. Let $\{x_n\}_{n \in \Z}$ be any sequence such that $x_{i+1}-x_i=h$ and $z_i \in [x_{i},x_{i+1}]$ be any points. The following is true:
\[
\left \vert h\sum_{i \in \Z}f(z_i)-\int f(x)dx \right \vert \leq h \int \left \vert f'(x) \right \vert dx.
\]
\end{lemma}

\subsection{Hermite Ensemble}
For any $\beta>0$ and small $\varepsilon>0$ and $k\in\N$, let $t$ and $ \gamma_{\varepsilon,\beta,k}$ be such that $t_{\varepsilon,\beta,k}\leq t\leq \gamma_{\varepsilon,\beta,k}n^{2/3}$ {($t_{\varepsilon,\beta,k}, \gamma_{\varepsilon,\beta,k}$ to be made precise later)}. For this subsection, we consider the matrix $\wih{H}_{n,\beta,k(p+4)}$ where $p:= \left\lceil \l(2\sqrt{t}+2\sech\l(\sqrt{t}\r)\r)\frac{n^{1/3}}{\sqrt{t}}\right\rceil$ (recall the definition of $\wih{H}_{n,\beta,p}$ from the beginning of this section). Denote $p_k:=k(p+4)$. We first obtain lower bounds for the right tails of $\lambda_{k}\l(\wih{H}_{n,\beta,p_k}\r)$.

For $i\in[k]$, we choose $k$ non-negative vectors $v_i$ such that $v_i=(v_i(0),v_i(1),\dots,v_i(n)) \in \R^{n+1}$ {(these vectors depend on $n,t$, $\varepsilon$ and $k$)}. The first vector is defined as $v_1\l(j\r)=0$ for all $j> p$ and $v_1\l(0\r)=0$. The non-zero entries are described later. For $i\in [k-1]$, define $v_{i+1}$ as $v_{i+1}(j+i(p+2))=v_{i}(j+(i-1)(p+2))$, if $j\in\{1,2,\dots,p\}$, otherwise the entries are $0$. In other words, we shift the entries and form new vectors from $v_1$. As defined, in each vector $v_i\in\R^{n+1}$, only $p$ many entries are non-zero. Note that as the matrix $\wih{H}_{n,\beta,p_k}$ is a tridiagonal matrix and the supports of $v_i$s are far enough (distance 3 apart), we have $\< \wih{H}_{n,\beta,p_k}v_i,v_j\>=0$ for $i\neq j$. As a result, 
\begin{align}\label{eq: lower bound on lambda_k}
\lambda_{k}\l(\wih{H}_{n,\beta,p_k}\r)\geq \min\limits_{i\in[k]}\frac{\< \wih{H}_{n,\beta,p_k}v_i,v_i\>}{\<v_i,v_i\>}.
\end{align}

We now show that for all  $t_{\varepsilon,\beta,k}\leq t\leq \gamma_{\vep,\beta,k}n^{2/3}$ and any $i\in[k]$, we have
\begin{align*}
    \< \wih{H}_{n,\beta,p_k}v_i,v_i\>\geq\l(2\sqrt{n}+tn^{-1/6}\r) \<v_i,v_i\>
\end{align*}
with probability at least $\exp\l(-\l(\frac{2\beta}{3}+\varepsilon\r)t^{3/2}\r)$. Due to the independence of $\< \wih{H}_{n,\beta,p_k}v_i,v_i\>$, we have that $\lambda_{k}\l(\wih{H}_{n,\beta,p_k}\r)\geq 2\sqrt{n}+tn^{-1/6}$ with probability at least $\exp\l(-\l(\frac{2\beta}{3}+\varepsilon\r)kt^{3/2}\r)$.

Fix any $i\in[k]$.
Note that $\< \wih{H}_{n,\beta,p_k}v_i,v_i\>$ is a Gaussian random variable $\mu_i+\sigma_i Z$, where $Z\sim N\l(0,1\r)$. Using the fact $\sqrt{n-j}\geq \sqrt{n}-\frac{j+1}{2\sqrt{n}}$ for $j\leq 2\sqrt{n}-1$ (choose $n_{\varepsilon}$ such that $p_k\leq 2\sqrt{n}-1$ for all $n\geq n_{\vep}$), we have that the mean $\mu$ of $\< \wih{H}_{n,\beta,p_k}v_i,v_i\>$ is given by
\begin{align}\label{eq: mean lower bound for right tail lower bound}
    \mu&= 2\sum\limits_{j=1}^{p_k-1}\sqrt{n-j}\ v_i\l(j\r)v_i\l(j+1\r)\nonumber\\
    &\geq \sum\limits_{j=1}^{p_k-1}\l(\sqrt{n}-\frac{j+1}{2\sqrt{n}}\r)\l(v_i^2\l(j\r)+v_i^2\l(j+1\r)-\l(v_i\l(j+1\r)-v_i\l(j\r)\r)^2\r)\nonumber\\
    &\geq 2\sqrt{n}\<v_i,v_i\>-\sqrt{n}\sum\limits_{j=0}^{p_k}\l(v_i\l(j+1\r)-v_i\l(j\r)\r)^2-\frac{1}{\sqrt{n}}\sum\limits_{j=1}^{p_k} \l(j+1\r) v_i^2\l(j\r) .
\end{align}
Note that the factor $\sqrt \beta$ does not appear in the expression for  $\mu$ as it gets canceled with the $\frac{1}{\sqrt{\beta}}$ in front of $\wih{H}_{n,\beta}$. Further, the first term in the second inequality should yield $2 \sqrt n \langle v_i,v_i \rangle-\sqrt{n}v_i^2(1).$ The extra $-\sqrt{n}v_i^2(1)$ is adjusted in the second sum on the third line by considering the range of $j$ starting from $0$. 

An important observation is that, due to the way the vectors $v_i$ are defined, the term 
\begin{align}\label{eq: i dep term}
    \frac{1}{\sqrt{n}}\sum\limits_{j=1}^{p_k} \l(j+1\r) v_i^2\l(j\r)
\end{align} in the third inequality above is the only term that varies with $i$. We later show that the contribution due to this term is negligible and the same lower bound on the probability works for every $i$. This is highlighted again later. This is the crucial difference from the left tail upper bound in \cite{BBBK25}, where the contribution due to this term is not negligible. This prevents us from using the same trick to obtain sharp upper bound on left tails as well.

Variance of $\< \wih{H}_{n,\beta,p_k}v_i,v_i\>$, denoted as $\sigma^2$ is given by,
\begin{align}\label{eq: variance eq hermite}
    \sigma^2&=
    \frac{4}{\beta}\sum\limits_{j=1}^{p_k}v_i^4\l(j\r)-\frac{1}{\beta}\sum\limits_{j=0}^{p_k}\l(v^2_i\l(j+1\r)-v_i^2\l(j\r)\r)^2.
\end{align}

Note that in the above equation, due to the way the vectors $v_i$ are defined, none of the terms on the right hand side depend on $i$.
Note that $\mu\leq 2\sqrt{n}\<v_i,v_i\>$. We have,
\begin{align*}
    \P\l(\mu+\sigma Z\geq\l(2\sqrt{n}+tn^{-1/6}\r)\<v_i,v_i\>\r)\geq \P\l(\rho+\sigma Z\geq\l(2\sqrt{n}+tn^{-1/6}\r)\<v_i,v_i\>\r) 
\end{align*} 
where $\rho<\mu$. Using the notation $\Delta\l(v\l(k\r)\r)=v\l(k+1\r)-v\l(k\r)$, we consider the event (taking lower bound in \eqref{eq: mean lower bound for right tail lower bound} as $\rho$),
\begin{equation}
\label{eq: Gaussian for right tail lower bound}
    2\sqrt{n}\<v_i,v_i\>-\sqrt{n}\sum\limits_{j=0}^{p_k}\l(\Delta\l(v_i\l(j\r)\r)\r)^2-\frac{1}{\sqrt{n}}\sum\limits_{j=1}^{p_k} \l(j+1\r) v_i^2\l(j\r) +\sigma_i Z\geq \l(2\sqrt{n}+tn^{-1/6}\r)\<v_i,v_i\>.
\end{equation}

 Let 
 \begin{align*}
 f_{t}\l(x\r)=\begin{cases}
 \sech\l(x-\sqrt{t}-\sech\l(\sqrt{t}\r)\r) &\text{ if } \sech\l(\sqrt{t}\r)\leq x\leq 2\sqrt{t}+\sech\l(\sqrt{t}\r).\\
 x &\text{ if } 0 \leq x \leq \sech(\sqrt{t})\\
  2\sqrt{t}+2\sech(\sqrt{t})-x &\text{ if } 2\sqrt{t}+\sech(\sqrt{t}) \leq x \leq 2\sqrt{t}+2\sech(\sqrt{t}).
 \end{cases}
 \end{align*}
 In words, for $0 \leq x\leq \sech\l(\sqrt{t}\r)$ and $2\sqrt{t}+\sech\l(\sqrt{t}\r)\leq x\leq 2\sqrt{t}+2\sech\l(\sqrt{t}\r)$, we use linear interpolation (of slope $1$) to get a continuous function with compact support. A similar function which is maximizer of a certain expression shows up in the calculations of lower bounds for right tails of $TW_{\beta}$ in \cite{RRV}. The linear part ensures that $\sqrt{n}v_1^2(1),\sqrt{n}v_1^2(p)$ terms (and the corresponding term for vectors $v_i,i\neq 2$) in \eqref{eq: mean lower bound for right tail lower bound} do not blow up. Denote $\tau_1=\l\lfloor\frac{\sech\l(\sqrt{t}\r)n^{1/3}}{\sqrt{t}}\r\rfloor$ and $\tau_2=\left\lceil\frac{\l(2\sqrt{t}+\sech\l(\sqrt{t}\r)\r)n^{1/3}}{\sqrt{t}}\right\rceil$ and $g(t)=2\sqrt{t}+2\sech\l(\sqrt{t}\r)$. Following \cite{RRV}, we define $v_1\l(j\r)=f_{t}\l(\frac{j\sqrt{t}}{n^{1/3}}\r)$ for $j\leq p$, and $v_1\l(j\r)=0$ otherwise.
By the definition of $v_1$ we have,
\begin{align}\label{eq: l2 norm of vector }
\sum\limits_{j=1}^{n}v_1^2\l(j\r)&=
\sum\limits_{j=1}^{\tau_1}\frac{j^2t}{n^{2/3}} +\sum\limits_{j=\tau_2}^{p}\l(g(t)-\frac{j\sqrt{t}}{n^{1/3}}\r)^2 +\sum\limits_{j= \tau_1+1}^{\tau_2-1}\sech^2\l(\frac{j\sqrt{t}}{n^{1/3}}-\sqrt{t}-\sech\l(\sqrt{t}\r)\r).
 \end{align}

Our goal is to  approximate the infinite sums involved in the expressions of $\mu$ and $\sigma$ with infinite integrals by using Lemma \ref{lem:aprox_inf_int}. The main difference here from the approach in \cite{BBBK25} is that we need to do it for all the $k$ vectors. But as the vectors are shifts of $v_1$, we analyse mostly $v_1$. The analysis of the terms dependent on $i$ are explicitly highlighted.

To this end we first observe that as $\sech(x)$ decays exponentially, for any $\varepsilon>0$, one can choose $n_{\varepsilon}, t_{\varepsilon},\gamma(\vep)$ such that for all $n\geq n_{\varepsilon}$ and $t_{\varepsilon}\leq t\leq \gamma(\vep)n^{\frac{2}{3}}$, we have that $\sum\limits_{j=1}^{\tau_1}\frac{j^2t}{n^{2/3}} +\sum\limits_{j=\tau_2}^{p}\l(g(t)-\frac{j\sqrt{t}}{n^{1/3}}\r)^2\leq \varepsilon \frac{n^{1/3}}{\sqrt{t}}$.
For any $\varepsilon>0$, one can choose $M_{\varepsilon}>0$ such that 
\begin{align}\label{eq: sech tail integrals}
    \int_{-\infty}^{-M_{\varepsilon}}\sech^2(x)dx,\ \int_{M_{\varepsilon}}^{\infty}\sech^2(x)dx\leq \varepsilon.
\end{align}

 This allows us to write, for large enough $t$ (in particular, $\sqrt{t}>M_\vep$) and $n$, we have 
 \begin{align*}
     \sum\limits_{j\geq 0}^{\infty}\sech^2\l(\frac{-j\sqrt{t}}{n^{1/3}}-M_{\varepsilon}\r)+\sech^2\l(\frac{j\sqrt{t}}{n^{1/3}}+M_{\varepsilon}\r)\leq \varepsilon \frac{n^{1/3}}{\sqrt{t}}.
 \end{align*} 
 Here we have used Lemma \ref{lem:aprox_inf_int} and \eqref{eq: sech tail integrals}. Hence we also have
 \begin{align*}
 \sum\limits_{j\geq 0}^{\infty}\sech^2\l(\frac{-j\sqrt{t}}{n^{1/3}}-\sqrt{t}\r)+\sech^2\l(\frac{j\sqrt{t}}{n^{1/3}}+\sqrt{t}\r)\leq \varepsilon \frac{n^{1/3}}{\sqrt{t}}.    
 \end{align*}
  This implies by applying Lemma \ref{lem:aprox_inf_int} again that, for large enough $n$ and $t$, 
 \begin{align}\label{eq: l2 norm calc1}
     \l|\int_{-\infty}^{\infty}\sech^2\l(x\r)dx-\frac{\sqrt{t}}{n^{1/3}}\sum\limits_{j=\tau_1+1}^{\tau_2-1}\sech^2\l(\frac{j\sqrt{t}}{n^{1/3}}-\sqrt{t}-\sech\l(\sqrt{t}\r)\r)\r|\leq \varepsilon.
 \end{align}
 By repeating the exact same argument for $\sum\limits_{j=1}^{n}v_1^4\l(j\r)$, we have
 \begin{align}\label{eq: right tail L2,L4 norm}
\l|\sum\limits_{j=1}^{n}v_1^2\l(j\r)-2\frac{n^{1/3}}{\sqrt{t}}\r|\leq \frac{n^{1/3}}{\sqrt{t}}\varepsilon\quad \quad \l|\sum\limits_{j=1}^{n}v_1^4\l(j\r)-\frac{n^{1/3}}{\sqrt{t}}\l(\frac{4}{3}\r)\r|\leq \frac{n^{1/3}}{\sqrt{t}}\varepsilon .
 \end{align}

Note that above bounds hold for any $v_i$, due to the definition of $v_i$ in terms of $v_1$.

Choosing $\sqrt{t}>M_{\varepsilon}$ and  by choosing $t,n$ large enough, by similar argument using Lemma \ref{lem:aprox_inf_int} it can be checked that
\begin{align*}
    \l|\sum\limits_{j=\tau_1+1}^{\tau_2-2}\l(\Delta\l(v_1\l(j\r)\r)\r)^2-\frac{\sqrt{t}}{n^{1/3}}\int_{-\infty}^{\infty}\l(\sech'\l(x\r)\r)^2dx\r|\leq \varepsilon\frac{\sqrt{t}}{n^{1/3}}.
\end{align*}
Hence we have

    \begin{equation}
    \label{eq: right tail derivative norm}
\l|\sum\limits_{j=0}^{n}\l(\Delta\l(v_1\l(j\r)\r)\r)^2-\frac{\sqrt{t}}{n^{1/3}}\l(\frac{2}{3}\r)\r|\leq\frac{\sqrt{t}}{n^{1/3}}\varepsilon.
    \end{equation}
Note that above bounds hold for any $v_i$, due to the definition of $v_i$ in terms of $v_1$.

Next consider $\sum\limits_{j=1}^{p}jv_1^2\l(j\r)$.  Elementary calculations (see \cite{BBBK25}[Equations $(3.9)$-$(3.10)$] for more details) give us, 

{
   
    \begin{align}\label{eq: kv(k)^2 final bound}
        \l|\sum\limits_{j=1}^{p}jv_1^2\l(j\r)-\frac{n^{2/3}}{\sqrt{t}}2\r|&\leq \l|\sum\limits_{j=\tau_1+1}^{\tau_2-1}j\sech^2\l(\frac{j\sqrt{t}}{n^{1/3}}-\sqrt{t}-\sech\l(\sqrt{t}\r)\r)-\frac{n^{2/3}}{\sqrt{t}}2\r|+\frac{n^{2/3}}{\sqrt{t}}2\varepsilon\nonumber\\
        &\leq \frac{n^{2/3}}{\sqrt{t}}3\varepsilon
    \end{align}
}

and also
    \begin{equation}
    \label{eq: right tail derivative square norm}
\l|\sum\limits_{j=0}^{p}\l(\Delta\l(v_1^2\l(j\r)\r)\r)^2-\frac{\sqrt{t}}{n^{1/3}}\int_{-\infty}^{\infty}\l(2FF'\l(x\r)\r)^2dx\r|\leq\frac{\sqrt{t}}{n^{1/3}}\varepsilon,
    \end{equation}
    where $F(x)=\sech(x).$

We observe that the bound \eqref{eq: right tail derivative square norm} (used in bounding the variance term as computed in \eqref{eq: variance eq hermite}) gives,
\begin{equation}
    \label{eq: right tail derivative square norm v_i}
\l|\sum\limits_{j=(i-1)(p+2)}^{i(p+2)}\l(\Delta\l(v_i^2\l(j\r)\r)\r)^2-\frac{\sqrt{t}}{n^{1/3}}\int_{-\infty}^{\infty}\l(2FF'\l(x\r)\r)^2dx\r|\leq\frac{\sqrt{t}}{n^{1/3}}\varepsilon.
    \end{equation}

As remarked before, in \eqref{eq: mean lower bound for right tail lower bound} only the last term, that is $\frac{1}{\sqrt{n}}\sum\limits_{j=1}^{p_k} \l(j+1\r) v_i^2\l(j\r) $, varies with $i$.

Note that
\begin{align*}
    \sum\limits_{j=1}^{p_k} \l(j+1\r) v_i^2\l(j\r)&= \sum\limits_{j=1}^{p} \l(j+1+(i-1)(p+2)\r) v_1^2\l(j\r)\\
    &=\sum\limits_{j=1}^{p} (j+1)v_1^2\l(j\r)+(i-1)(p+2) \sum\limits_{j=1}^{p}v_1^2\l(j\r).
\end{align*}

It follows from \eqref{eq: kv(k)^2 final bound}, \eqref{eq: right tail L2,L4 norm} and the value of $p$ that,
\begin{align}\label{eq: kv(k)^2 final bound v_i}
    \l|\sum\limits_{j=1}^{p_k}(j+1)v_i^2\l(j\r)-\frac{n^{2/3}}{\sqrt{t}}(2+4(i-1))\r|\leq \frac{n^{2/3}}{\sqrt{t}}5\varepsilon.
\end{align}

Note that from \eqref{eq: Gaussian for right tail lower bound} we need a lower bound for the probability of the event 
\begin{equation}
\label{eq: Gaussian event for right tail lower bound}
    Z_i \geq \frac{\sqrt{n}\sum\limits_{j=0}^{p_k}\l(\Delta\l(v_i\l(j\r)\r)\r)^2+\frac{1}{\sqrt{n}}\sum\limits_{j=1}^{p_k} \l(j+1\r) v_i^2\l(j\r)+tn^{-1/6}\<v_i,v_i\>}{\sqrt{\frac{4}{\beta}\sum\limits_{j=1}^{p_k}v_i^4\l(j\r)-\frac{1}{\beta}\sum\limits_{j=0}^{p_k}\l(v_i^2\l(j+1\r)-v_i^2\l(j\r)\r)^2}}.
\end{equation}

Using the bounds from \eqref{eq: right tail L2,L4 norm},\eqref{eq: right tail derivative norm} and \eqref{eq: kv(k)^2 final bound v_i} we consider,
\begin{align*}
    Z_i\geq \frac{\sqrt{n}\frac{\sqrt{t}}{n^{1/3}}\l(\frac{2}{3}+\varepsilon\r)+\frac{1}{\sqrt{n}}\l(\frac{n^{2/3}}{\sqrt{t}}\l(2+4(i-1)+\varepsilon\r)+\frac{n^{1/3}}{\sqrt{t}}\l(2+\varepsilon\r)\r)+\frac{t}{n^{1/6}}\frac{n^{1/3}}{\sqrt{t}}\l(2+\varepsilon\r)}{\sqrt{\frac{4n^{1/3}}{\beta\sqrt{t}}\l(\frac{4}{3}-\varepsilon\r)-\frac{1\sqrt{t}}{\beta n^{1/3}}\l(c+\varepsilon\r)}},
\end{align*}
where $c= \int_{-\infty}^{\infty}\l(2FF'\l(x\r)\r)^2dx$ and $F(x)=\sech(x).$ Note that we have taken a slightly larger value for the numerator and a lower value for the denominator as we are interested in the lower bound.

We highlight a crucial point on which the whole proof is based. The second term in the numerator, the term that corresponds to the term depending on $i$ in \eqref{eq: i dep term}, is of lower order compared to the other terms in the numerator. For this reason, the dependence on $i$ and also the shift of the vectors do not affect the argument.

Hence we find that for any $\vep>0,$ there exists $n_{\varepsilon,\beta,k},t_{\varepsilon,\beta,k}$ such that for all $n\geq n_{\varepsilon,\beta,k}$ and $t\geq t_{\varepsilon,\beta,k}$, the probability of the event that we want is lower bounded by the probability of the following event
\[
Z_i \geq \frac{2\sqrt{\frac \beta3}t^{3/4}\l(1+\vep\r)}{\l(1-\vep\r)},
\]
where $Z_i$ is a $N\l(0,1\r)$ variable. Hence, for any $\vep>0$, there exist $n_{\varepsilon,\beta,k}$ and $t_{\varepsilon,\beta,k}$ such that for all $n\geq n_{\varepsilon,\beta}$ and $t_{\varepsilon,\beta}\leq t$, we have
\begin{align*}
    \< \wih{H}_{n,\beta,p_k}v_i,v_i\>\geq\l(2\sqrt{n}+tn^{-1/6}\r) \<v_i,v_i\>
\end{align*}
 with probability at least $\exp\l(-\l(\frac{2\beta}{3}+\varepsilon\r)t^{3/2}\r)$. Due to the independence of $\< \wih{H}_{n,\beta,p_k}v_i,v_i\>$ and \eqref{eq: lower bound on lambda_k}, we have that 
\begin{align*}
\mathbb{P}\l(\lambda_{k}\l(\wih{H}_{n,\beta,p_k}\r)\geq 2\sqrt{n}+tn^{-1/6}\r)\geq \exp\l(- \frac{2\beta k}{3}t^{3/2}(1+\varepsilon)\r).
\end{align*}

Further, note that by Weyl's inequality we have
\[
\l\{\lambda_{k} \l(\wih{H}_{n,\beta,p_k} \r) \geq 2 \sqrt{n}+t n^{-1/6}\r\} \subset \l\{ \lambda_{k} \l(H_{n, \beta}\r) \geq 2 \sqrt{n}+t n^{-1/6}-{\lambda_{\max} \l(\wih{H}_{n,\beta,p_k}-H_{n, \beta} \r)}\r\}.
\]
We take $x=\varepsilon t n^{1/3}$ and apply Proposition \ref{eq: estimate of off diagonals of the difference}, to get for sufficiently large $n$
\[
\P \l( \lambda_{k} \l(H_{n, \beta}\r) \geq 2 \sqrt{n}+ \l(1-\vep\r) tn^{-1/6}\r)+C n\l(3e^{-x}+2e^{-\frac{\beta\l(n-p_k\r)}{2}}\r)\geq \P \l( \lambda_k\l(\wih{H}_{n,\beta,p_k}\r)  \geq 2 \sqrt{n}+t n^{-1/6} \r).
\]
By elementary calculations one can show that for any $\beta>0$ and any small $\varepsilon>0$ and $k\in\N$, there exist $n_{\varepsilon,\beta,k},\gamma_{\vep,\beta,k}$ and $t_{\varepsilon,\beta,k}$ such that for $n\geq n_{\varepsilon,\beta,k}$ and $t_{\varepsilon,\beta,k}\leq t\leq \gamma_{\vep,\beta,k}n^{2/3}$, we have that
\begin{align*}
    \mathbb{P}\l(\lambda_{k}\l(H_{n,\beta}\r)\geq 2\sqrt{n}+tn^{-1/6}\r)\geq \exp\l(- \frac{2\beta k}{3}t^{3/2}\l(1+\varepsilon\r)\r).
\end{align*}

\subsection{Laguerre Ensemble}\label{subsec: Laguerre lower for left}

Fix $\beta>0$ and fix small $\varepsilon>0$ and $k\in\N$.
Let $m=m\l(n\r)\geq n$ be such that for all $n\geq n_{\varepsilon,\beta,M,k}$ (to be made precise later), we have $\frac{m}{n}\leq M$ for some $M<\infty$. We denote $\gamma:=m/n$ and $a=\l(1+\sqrt{\gamma}\r)^2$ and $b=\gamma^{-1/6}\l(1+\sqrt{\gamma}\r)^{4/3}.$ We consider the matrix $\wih{L}_{n,m,\beta,k(p+4)}$ where $p:= \left\lceil \l(2\sqrt{t}+2\sech\l(\sqrt{t}\r)\r)\frac{n^{1/3}\gamma^{1/4}}{\sqrt{tb}}\right\rceil$ (recall the definition of $\l(\wih{L}_{n,m,\beta,p}\r)$ from Proposition \ref{eq: estimate of off diagonals of the difference}). Denote $p_k=k(p+4)$. We obtain a lower bound for the probability of the event $\lambda_{k}\l(\wih{L}_{n,m,\beta,p_k}\r)\geq an\l(1+\frac{tbn^{-2/3}}{a}\r)$. Here $t_{\varepsilon,\beta,M,k}\leq t\leq \rho_{\vep,\beta,M,k}n^{2/3}$. The values $t_{\varepsilon,\beta,M,k},\rho_{\vep,\beta,M,k}$ will be made precise later. As in the case of Hermite ensemble, for $\varepsilon>0$ we construct $k$ non-negative vectors $v_1,v_2,\dots,v_k$ (these vectors depend on $n,m,t,k$ and $\varepsilon$). We first define $v_1$ such that $v_1\l(i\r)=0$ for all $i> p$ and $v_1\l(0\r)=0$. The non-zero entries are described later. For $i\in [k-1]$, define $v_{i+1}$ as $v_{i+1}(j+i(p+2))=v_{i}(j+(i-1)(p+2))$, if $j\in\{1,2,\dots,p\}$, otherwise the entries are $0$. In other words, we shift the entries and form new vectors from $v_1$. As defined, in each vector $v_i\in\R^{n+1}$, only $p$ many entries are non-zero. Note that as the matrix $\wih{L}_{n,m,\beta,p_k}$ is a tridiagonal matrix and the supports of $v_i$s are far enough (distance 3 apart), we have $\< \wih{L}_{n,m,\beta,p_k}v_i,v_j\>=0$ for $i\neq j$. As a result, 
\begin{align}\label{eq: lower bound on lambda_k laguerre}
\lambda_{k}\l(\wih{L}_{n,m,\beta,p_k}\r)\geq \min\limits_{i\in[k]}\frac{\< \wih{L}_{n,m,\beta,p_k}v_i,v_i\>}{\langle v_i,v_i\rangle}.
\end{align}

We define $v_i$ and show that for all $i\in[k]$ and for all $n\geq n_{\varepsilon,\beta,M,k}$ and  $t_{\varepsilon,\beta,M,k}\leq t\leq \rho_{\vep,\beta,M,k}n^{2/3}$, we have
\begin{align}\label{eq: Laguerre before KMT}
    \langle \wih{L}_{n,m,\beta,p_k}v_i,v_i\rangle\geq an\l(1+\frac{tbn^{-2/3}}{a}\r)\langle v_i,v_i\rangle
\end{align}
with probability at least $\exp\l(- \frac{2\beta}{3}t^{3/2}\l(1+\varepsilon\r)\r)$. Due to the independence of $\langle \wih{L}_{n,m,\beta,p_k}v_i,v_i\rangle$, we get that
\begin{align}\label{eq: }
\lambda_{k}\l(\wih{L}_{n,m,\beta,p_k}\r)\geq an\l(1+\frac{tbn^{-2/3}}{a}\r),
\end{align}
with probability at least $\exp\l(- \frac{2\beta k}{3}t^{3/2}\l(1+\varepsilon\r)\r)$.
Note that $\wih{L}_{n,m,\beta,p_k}$ can be written as a sum of two matrices $L'$ and $L''$, where $L''$ is a tridiagonal symmetric matrix which contains entries which are linear combinations of $\zeta_k^2,\l(\zeta'_k\r)^2, \zeta_{k+1}\zeta'_k$ in the top $p_k\times p_k$ submatrix and zeros elsewhere. $L'$ contains random variables which are Gaussian in the top $p_k\times p_k$ submatrix. We first show that for the range $t_{\varepsilon,\beta,M,k}\leq t\leq \rho_{\vep,\beta,M,k}n^{{2/3}}$ and for $i\in[k]$,
\begin{align*}
    \langle L'v_i,v_i\rangle\geq an\l(1+\frac{tbn^{-2/3}}{a}\r)\langle v_i,v_i\rangle
\end{align*}
with probability at least $\exp\l(- \frac{2\beta}{3}t^{3/2}\l(1+\varepsilon\r)\r)$. We take $v_1\l(j\r)=0$ for $j> p$ and define $v_i$ as described earlier. We also use $\sqrt{n-j}\geq \sqrt{n}-\frac{j+1}{2\sqrt{n}}$ for $j\leq 2\sqrt{n}-1$. We choose $n_{\vep}$ such that $p_k\leq 2\sqrt{n}-1$ for all $n\geq n_{\vep}$.
Now $\langle L'v_i,v_i\rangle$ is a Gaussian random variable with mean $\mu_i$ and variance $\sigma_i^2$ given by,
\begin{align*}
    \mu_i&= \sum\limits_{j=1}^{p_k} \l(\l(m+1-j\r)+\l(n-j\r)\r)v_i^2\l(j\r) + 2\sum\limits_{j=1}^{p_k-1} \l(\sqrt{\l(m-j\r)\l(n-j\r)}\r)v_i\l(j\r)v_i\l(j+1\r)\nonumber\\
    &\geq \sum\limits_{j=1}^{p_k} \l(m+n-2j\r)v_i^2\l(j\r) +\sum\limits_{j=1}^{p_k-1} \l(\sqrt{m}-\frac{j{+1}}{2\sqrt{m}}\r)\l(\sqrt{n}-\frac{j+{1}}{2\sqrt{n}}\r)\l(v_i^2\l(j\r)+v_i^2\l(j+1\r)-\l(\Delta\l(v_i\l(j\r)\r)\r)^2 \r)\nonumber\\
    &\geq \l(\sqrt{m}+\sqrt{n}\r)^2\langle v_i,v_i\rangle-\sum\limits_{j=1}^{p_k} \l(2j+\frac{\l({m}+{n}\r)\l(j+1\r)}{\sqrt{mn}}\r)v_i^2\l(j\r)-\sum\limits_{j=0}^{p_k}\l(\sqrt{mn}+\frac{\l(j+1\r)^2}{4\sqrt{mn}}\r)\l(\Delta\l(v_i\l(j\r)\r)\r)^2. 
\end{align*}

Note that in the above inequality, due to the way the vectors $v_i$ are defined, only the first term on the right hand side does not depend on $i$.

\begin{align*}
    \sigma_i^2&=  \frac{2}{\beta}\sum\limits_{j=0}^{p_k-1}\l({\sqrt{m-j}\ v_i^2\l(j+1\r)}+\sqrt{n-j}\ v_i\l(j\r)v_i\l(j+1\r)\r)^2\\  
    &+\frac{2}{\beta}\sum\limits_{j=1}^{p_k} \l({\sqrt{n-j}\ v_i^2\l(j\r)}+\sqrt{m-j}\ v_i\l(j\r)v_i\l(j+1\r)\r)^2. 
\end{align*}
We are interested in the event 
\begin{align*}
    \mu_i+\sigma_i Z_i \geq \l(\sqrt{m}+\sqrt{n}\r)^2\langle v_i,v_i\rangle + tbn^{1/3}\sum\limits_{j=1}^{p_k} v_i^2\l(j\r).
\end{align*}
{Note that $\mu_i\leq \l(\sqrt{m}+\sqrt{n}\r)^2\langle v_i,v_i\rangle$.} As in the case of Hermite ensemble, we can take the lower bound of $\mu_i$ to bound the probability of the above event. It suffices to find a lower bound for the probability of the event,
\begin{equation}\label{eq: Laguerre event lower bound}
    \sigma_i Z_i\geq  tbn^{1/3}\sum\limits_{j=1}^{p_k} v_i^2\l(j\r) +\sum\limits_{j=1}^{p_k} \l(2j+\frac{\l({m}+{n}\r)\l(j+1\r)}{\sqrt{mn}}\r)v_i^2\l(j\r)+\sum\limits_{j=0}^{p_k}\l(\sqrt{mn}+\frac{\l(j+1\r)^2}{4\sqrt{mn}}\r)\l(\Delta\l(v_i\l(j\r)\r)\r)^2.
\end{equation}

We also have that 
\begin{align*}
    \sigma_i^2
    &\geq \l(\frac{2}{\beta}\l(m+n-2p_k\r)\r)\l(\sum\limits_{j=1}^{p_k}2v_i^4\l(j\r)-\frac{1}{2}\sum\limits_{j=0}^{p_k}\left(\Delta\l(v_i^2\l(j\r)\r)\right)^2\r)\nonumber\\
    &\quad +\frac{2}{\beta}\sqrt{\l(m-p_k\r)\l(n-p_k\r)}\l(\sum\limits_{j=1}^{p_k}4v_i^4\l(j\r)-\sum\limits_{j=0}^{p_k}\l(\Delta\l(v_i^2\l(j\r)\r)\r)^2-\sum\limits_{j=0}^{p_k}\l(\Delta\l(v_i\l(j\r)\r)\r)^2\l(v_i^2\l(j\r)+v_i^2\l(j+1\r)\r)\r)\\
    &= \frac{4}{\beta}\l(\sqrt{m-p_k}+\sqrt{n-p_k}\r)^2\l(\sum\limits_{j=1}^{p_k}v_i^4\l(j\r)-\frac{1}{4}\sum\limits_{j=0}^{p_k}\l(\Delta\l(v_i^2\l(j\r)\r)\r)^2\r)\nonumber\\
    &-\frac{2}{\beta}\sqrt{\l(m-p_k\r)\l(n-p_k\r)}\sum\limits_{j=0}^{p_k}\l(\Delta\l(v_i\l(j\r)\r)\r)^2\l(v_i^2\l(j\r)+v_i^2\l(j+1\r)\r)\nonumber.
\end{align*}

Note that in the above inequality, due to the way the vectors $v_i$ are defined, none of the terms on the right hand side depend on $i$.

We denote $\nu$ to be the RHS of the above inequality. Let $f_{t}\l(x\r)$ be the function as in the Hermite case. We define $v_1\l(j\r)=f_{t}\l(\frac{j\sqrt{tb}}{\gamma^{1/4}n^{1/3}}\r)$ for $j\leq p$ and $0$ otherwise. Using this vector, we define $v_i$ for $i\geq 2$. By the choice of vector $v_1$ and comparing sums with integrals as we did in the Hermite case, one can check that for any given $\varepsilon>0$, one can choose $n_{\varepsilon}, \rho(\vep,M)$ and $t_{\varepsilon}$ such that for all $t_{\varepsilon}\leq t \leq \rho(\vep,M)n^{2/3}$ and for all $n\geq n_{\varepsilon}$ with $\frac{m}{n}\leq M$ for some $M<\infty$, the following bounds hold. Recall that $F(x)=\sech(x).$

\begin{align*}
\l|\sum\limits_{j=1}^{n}v_1^2\l(j\r)-2\frac{n^{1/3}\gamma^{1/4}}{\sqrt{t}\sqrt{b}}\r|\leq \frac{n^{1/3}\gamma^{1/4}}{\sqrt{t}\sqrt{b}}\varepsilon\quad \quad \l|\sum\limits_{j=1}^{n}v_1^4\l(j\r)-\frac{n^{1/3}\gamma^{1/4}}{\sqrt{t}\sqrt{b}}\l(\frac{4}{3}\r)\r|\leq \frac{n^{1/3}\gamma^{1/4}}{\sqrt{t}\sqrt{b}}\varepsilon .
 \end{align*}

 \begin{align*}
\l|\sum\limits_{j=0}^{p_k}\l(\Delta\l(v_1\l(j\r)\r)\r)^2-\frac{\sqrt{t}\sqrt{b}}{n^{1/3}\gamma^{1/4}}\l(\frac{2}{3}\r)\r|\leq\frac{\sqrt{t}\sqrt{b}}{n^{1/3}\gamma^{1/4}}\varepsilon.
    \end{align*}

 \begin{align*}
\l|\sum\limits_{j=0}^{p_k}\l(\Delta\l(v_1^2\l(j\r)\r)\r)^2-\frac{\sqrt{t}\sqrt{b}}{n^{1/3}\gamma^{1/4}}\int_{-\infty}^{\infty}\l(2FF'\l(x\r)\r)^2dx\r|\leq\frac{\sqrt{t}\sqrt{b}}{n^{1/3}\gamma^{1/4}}\varepsilon.
    \end{align*}

    \begin{align*}
\l|\sum\limits_{j=0}^{p_k}\l(\Delta\l(v_1\l(j\r)\r)\r)^2\l(v_1^2\l(j\r)+v_1^2\l(j+1\r)\r)-\frac{\sqrt{t}\sqrt{b}}{n^{1/3}\gamma^{1/4}}\int_{-\infty}^{\infty}2\l(FF'\l(x\r)\r)^2dx\r|\leq\frac{\sqrt{t}\sqrt{b}}{n^{1/3}\gamma^{1/4}}\varepsilon.
    \end{align*}
Note that, due to the way $v_i$ are defined, all of the above bounds also hold for $v_i$. We also have,
\begin{align*}
\l|\sum\limits_{j=1}^{p_k}jv_1^2\l(j\r)-\frac{n^{2/3}\sqrt{\gamma}}{\sqrt{t}b}2\r|\leq \frac{n^{2/3}\sqrt{\gamma}}{\sqrt{t}b}\varepsilon.
\end{align*}
As a result we get,
 \begin{align*}
\l|\sum\limits_{j=1}^{p_k}jv_i^2\l(j\r)-\frac{n^{2/3}\sqrt{\gamma}}{\sqrt{t}b}(2+4(i-1))\r|\leq \frac{n^{2/3}\sqrt{\gamma}}{\sqrt{t}b}\varepsilon.
\end{align*}

We use the above bounds to get lower bound for the probability of the event 
\begin{align*}
    Z_i&\geq \frac{tbn^{1/3}\sum\limits_{j=1}^{p_k} v_i^2\l(j\r) +\sum\limits_{j=1}^{p_k} \l(2j+\frac{\l({m}+{n}\r)\l(j+1\r)}{\sqrt{mn}}\r)v_i^2\l(j\r)+\sum\limits_{j=0}^{p_k}\l(\sqrt{mn}+\frac{\l(j+1\r)^2}{4\sqrt{mn}}\r)\l(\Delta\l(v_i\l(j\r)\r)\r)^2}{\sqrt{\nu}}.
    \end{align*}

    Denoting 
    \begin{align*}
        \nu_{\ell}:= {\frac{4}{\beta}\l(\sqrt{m-p_k}+\sqrt{n-p_k}\r)^2\l(\frac{n^{1/3}\gamma^{1/4}}{\sqrt{t}\sqrt{b}}\l(\frac{4}{3}-\varepsilon\r)-\frac{1}{4}\frac{\sqrt{t}\sqrt{b}}{n^{1/3}\gamma^{1/4}}\l(\int_{-\infty}^{\infty}6\l(FF'\l(x\r)\r)^2dx+\varepsilon\r)\r)},
    \end{align*}
    one can check that $\nu\geq \nu_{\ell}$. As a result, we have that it is enough to find lower bound for the probability of the following event.
    \begin{align*}
    Z_i &\geq \frac{\l(tbn^{1/3}+\frac{\l(m+n\r)}{\sqrt{mn}}\r)\l(2+\varepsilon\r)\frac{n^{1/3}\gamma^{1/4}}{\sqrt{t}\sqrt{b}}+\l(2+\frac{\l({m}+{n}\r)}{\sqrt{mn}}\r)\frac{n^{2/3}\sqrt{\gamma}}{\sqrt{t}b}\l(2+4(i-1)+\varepsilon\r)}{\sqrt{\nu_{\ell}}}\\
    &+\frac{\l(\sqrt{mn}+\frac{\l(p_k+1\r)^2}{4\sqrt{mn}}\r)\frac{\sqrt{t}\sqrt{b}}{n^{1/3}\gamma^{1/4}}\l(\frac{2}{3}+\varepsilon\r)}{\sqrt{\nu_{\ell}}}.
\end{align*}

As in the Hermite case, the contribution due to the 
\begin{align*}
    \sum\limits_{j=1}^{p_k} \l(2j+\frac{\l({m}+{n}\r)\l(j+1\r)}{\sqrt{mn}}\r)v_i^2\l(j\r)+\sum\limits_{j=0}^{p_k}\l(\frac{\l(j+1\r)^2}{4\sqrt{mn}}\r)\l(\Delta\l(v_i\l(j\r)\r)\r)^2
\end{align*}
terms, which are the only terms that depend on $i$, are negligible compared to other terms in the numerator.

Hence, for any $\vep>0$ and for sufficiently large $t$ and $n$ our desired lower bound is the lower bound of probability of the event,
\[
Z_i \geq \frac{\gamma^{1/4}\sqrt{tb}n^{2/3}\frac 83\l(1+\vep\r)}{\frac{2}{\sqrt{\beta}}\l(\sqrt{m}+\sqrt{n}\r)\frac{\gamma^{1/8}n^{1/6}}{\l(tb\r)^{1/4}}\frac{2}{\sqrt{3}}\l(1-\vep\r)}.
\]
By putting the value of $b$ and $\gamma$ one gets the above event is equal to
\[
Z_i \geq \frac{2\sqrt{\beta}t^{3/4}}{\sqrt{3}}\frac{\l(1+\vep\r)}{\l(1-\vep\r)}.
\]
Applying lower bound for Gaussian tail we get that for any $\beta>0$ and small $\varepsilon>0$ and $M>0,k\in \N$, there exist $n_{\varepsilon,\beta,M,k},t_{\varepsilon,\beta,M,k}$ such that for the range $t_{\varepsilon,\beta,M,k}\leq t\leq \rho_{\vep,\beta,M,k}n^{2/3}$ and $n\geq n_{\varepsilon,\beta,M,k}$ with $m$ such that $\frac{m}{n}\leq M$, we have
\[\langle L'v_i,v_i\rangle\geq an\l(1+\frac{tbn^{-2/3}}{a}\r)\langle v_i,v_i\rangle\]
with probability at least $\exp\l(- \frac{2\beta}{3}t^{3/2}\l(1+\varepsilon\r)\r)$.

 Using Gershgorin circle theorem (Theorem $6.1.1$ of \cite{HJ}) and using the tail behaviour of normal random variable we get that the eigenvalues of the matrix $L''$ (introduced after \eqref{eq: Laguerre before KMT}) are located in an interval $\l(-2\varepsilon tn^{1/3},2\vep tn^{1/3}\r)$ with probability at least $1-6p_k\exp\l(-\varepsilon tn^{1/3}/2\r)$. Let us denote this event by $\ce_1.$ By Weyl's inequality and Gershgorin's circle theorem we have 
 \begin{align*}
\lambda_{k}\l(\wih{L}_{n,m,\beta,p}\r)\geq an\l(1+\frac{tbn^{-2/3}}{a}\r)-2\vep tn^{1/3}
 \end{align*}
with probability at least $\exp\l(- \frac{2\beta k}{3}t^{3/2}\l(1+\varepsilon\r)\r)-6p_k\exp\l(-\vep tn^{1/3}/2\r)$.

Further, note that by Weyl's inequality we have
\begin{align*}
&\l\{\lambda_{k} \l(\wih{L}_{n,m,\beta,p} \r) \geq an+{tbn^{1/3}}-2\vep tn^{1/3}\r\}\\
&\subset \l\{ \lambda_{k} \l(L_{n,m,\beta}\r) \geq an+{tbn^{1/3}}-2\vep tn^{1/3}-{\lambda_{\max} \l(\wih{L}_{n,m,\beta,p}-L_{n,m,\beta} \r)}\r\}.
\end{align*}
By taking $x=\vep t n^{1/3}$ and applying Proposition \ref{eq: estimate of off diagonals of the difference} we have that

\begin{align*}
&\mathbb{P}\l( \lambda_{k}\l(L_{n,m,\beta}\r)\geq an+t b n^{1/3}-\varepsilon t b n^{1/3}-2\vep tn^{1/3}\r)\\
&\geq \exp\l(- \frac{2\beta k}{3}t^{3/2}\l(1+\varepsilon\r)\r)-6p_k\exp\l(-\vep tn^{1/3}/2\r)-6p_k\exp\l(-\vep tn^{1/3}/2\r)-8p_ke^{-x}-4p_ke^{-\frac{\beta\l(n-p_k\r)}{2}}.
\end{align*}

Hence we have that for any $\beta>0$ and any small $\varepsilon>0$ and $M>0,k\in\N$, there exist $n_{\varepsilon,\beta,M,k}, \rho_{\vep,\beta,M,k}$ and $t_{\varepsilon,\beta,M,k}$ such that for $n\geq n_{\varepsilon,\beta,M,k}$ and $t_{\varepsilon,\beta,M,k}\leq t\leq \rho_{\vep,\beta,M,k}n^{2/3}$ with $m$ such that $\frac{m}{n}\leq M$, we have
\begin{align}\label{eq: Laguerre right tail lower bound final form}
\mathbb{P}\l( \lambda_{k}\l(L_{n,m,\beta}\r)\geq an+t b n^{1/3}\r)
\geq \exp\l(- \frac{2\beta k}{3}t^{3/2}\l(1+\varepsilon\r)\r).
\end{align}

\section{Upper bound for right tails}\label{sec:rtub}
\subsection{Hermite ensemble.} We first obtain sharp upper bound for the sum of top $k$ eigenvalues in $\beta=2$ case. To this end, we use the distributional identity between the eigenvalues of GUE ($\beta=2$ case of Hermite ensemble) and the passage times of Brownian last passage percolation.

\textbf{Brownian last passage percolation:}
Let $B_1, B_2, \dots, B_n$ be $n$ i.i.d. standard Brownian motions starting at zero ($B_i(0) = 0$ for all $1 \le i \le n$). For $t > 0$ and integer $k \in \{1, \dots, n\}$, consider a collection of $k$ disjoint upright paths $\{\pi_1, \dots, \pi_k\}$. For each $j \in \{1, \dots, k\}$, the path $\pi_j$ starts at $(0, j)$ and ends at $(t, n - k + j)$. Each path $\pi_j$ moves sequentially across the lines $j, j+1, \dots, n - k + j$. The jump times of path $\pi_j$ are given by the vector
\[
{t}^j = \left(t_{0}^j, t_1^j, t_{2}^j, \dots, t_{n-k+1}^j\right),
\]
such that
\[
0 = t_{0}^j \leq t_1^j \leq t_{2}^j \leq \dots \leq t_{n-k}^j \leq t_{n-k+1}^j = t, \quad \forall \, j \in \{1, \dots, k\}.
\]

To ensure that the paths $\pi_1, \dots, \pi_k$ are pairwise disjoint, the times must satisfy $t_i^{j+1} < t_{i}^j$ for $i\in[n-k]$. The passage time $L(\pi_j)$ along the individual path $\pi_j$ is given by the sum of Brownian increments:
\[
L(\pi_j) = \sum_{i=j}^{n-k+j} \left( B_i(t_{i-j+1}^j) - B_i(t_{i-j}^j) \right).
\]

The {$k$-disjoint path Brownian last passage time} from $(0,1)$ to $(t, n)$ is defined as:
\[
D_k\left((0,1) \to (t, n)\right) := \sup_{\{\pi_1,\pi_2,\dots,\pi_k\} } \sum_{j=1}^{k} L(\pi_j).
\]

 The connection between Brownian LPP and GUE eigenvalues is given by the following result from  \cite[Corollary 2]{AVMW13}.
 
\begin{proposition}[Corollary 2 of \cite{AVMW13}]
    For $\beta=2$, and $k\in\N$, the sum of $k$ largest eigenvalues in \eqref{Hermite density} has the same distribution as $D_k\left((0,1) \to (t, n)\right)/\sqrt{t}.$
\end{proposition}

Note that $D_k\left((0,1) \to (t, n)\right)$ possess the super-additive structure. That is,
\begin{align*}
    D_k({(0,1)\to (t, n)})+D_k({(t,n+1)\to (2t, 2n)})\leq D_k({(0,1)\to (2t, 2n)}).
\end{align*}
The two summands in the left hand side above are i.i.d. random variables. We also have that, due to the Brownian scaling,
\begin{align*}
    D_k({(0,1)\to (1, n)})\overset{d}{=}D_k({(0,1)\to (t, n)})/\sqrt{t}.
\end{align*}
Hence we have,
\begin{align*}
    \P\l(D_k({(0,1)\to (1, n)})\geq 2k\sqrt{n}+tkn^{-1/6}\r)^m\leq \P\l(D_k({(0,1)\to (m, mn)})\geq 2mk\sqrt{n}+tkmn^{-1/6}\r).
\end{align*}
 This gives us,
 \begin{align}\label{eq: Hermite LDP}
\frac{\log\l(\P\l(D_k({(0,1)\to (1, n)})\geq 2k\sqrt{n}+tkn^{-1/6}\r)\r)}{n}&\leq \frac{\log\l(\P\l(\frac{D_k({(0,1)\to (1, mn)})}{\sqrt{mn}}\geq 2k+tkn^{-2/3}\r)\r)}{mn}\notag \\
&=\frac{\log\l(\P\l(\frac{\sum\limits_{i=1}^k\lambda_i}{\sqrt{mn}}\geq 2k+tkn^{-2/3}\r)\r)}{mn},
 \end{align}
 where $\lambda_i$ are the top $k$ eigenvalues of GUE matrix of size $mn$.
 We apply the large deviation result of \cite{GGM12} for the top $k$ eigenvalues of GUE.
\cite[Theorem 2.10]{GGM12} gives that the joint law of top $k$ eigenvalues of GUE satisfies large deviation principle in the scale $n$ with good rate function given by 
\begin{align*}
    J^{(k)}(x_1,\dots,x_k)=\sum\limits_{i=1}^kJ_V(x_i)+k\alpha_{V,\beta}^{1}\quad \text{ if } x_1\geq x_2\geq\dots\geq x_k, 
\end{align*}
 with $J_V(x)=V(x)-2\int\log|x-y|d\mu_V(y)$. Here $V(x)= x^2/2$ and $\beta=2$. Also $\alpha_{V,\beta}^{1}=-\beta/2$ and $d\mu_V(y)=\frac{1}{2\pi}\sqrt{4-y^2}dy.$
Applying the contraction principle (see \cite{DZ98}), as sum is a continuous function, we get that $\sum\limits_{i=1}^k\lambda_i$ satisfies large deviation principle in scale $n$ with good rate function given by
\begin{align*}
    J_{\text{sum}}(s)=\inf\l\{\sum\limits_{i=1}^kJ(x_i):x_1\geq x_2\geq \dots\geq x_k\geq 2,\sum x_i=s\r\},
\end{align*}
where $J(x)=J_V(x)-1$ with $V(x)=x^2/2$. For GUE/GOE, $J_V(x)=  2\int\limits_{2}^{x}\sqrt{(y/2)^2-1}dy+1$. As $J(x)$ is convex and applying Jensen's inequality gives us that $J_{\text{sum}}(s)=kJ(s/k).$ Also $J(2+x)/x^{3/2}\rightarrow 4/3.$ Taking $m\rightarrow\infty$ in \eqref{eq: Hermite LDP} we get that for any given $\vep>0$, there exists $n_{k,\vep}$ so that for any $n\geq n_{k,\vep},$ we have
\begin{align*}
\P\l(D_k({(0,1)\to (1, n)})
\geq 2k\sqrt{n}+tkn^{-1/6}\r)\leq \exp\l(-\frac{4k}{3}t^{3/2}(1-\vep)\r).
\end{align*}
This implies
\begin{align*}
    \P\l(\lambda_k(H_{n,2})\geq 2\sqrt{n}+tn^{-1/6}\r)\leq \P\l(\sum\limits_{i=1}^k\lambda_i(H_{n,2})\geq 2k\sqrt{n}+ktn^{-1/6}\r)\leq\exp\l(-\frac{4k}{3}t^{3/2}(1-\vep)\r).
\end{align*}

Using the superposition-decimation result of \eqref{eq: GOE-GUE relation}, we have that
\begin{align}\label{eq: GUE/GOE superposition}
\text{even}\l(\text{GOE}_{n}\cup \text{GOE}_{n+1}\r)=\text{GUE}_n.
\end{align} 

As a result, we have 
\begin{align*}
\P\l(\lambda_k(H_{n,1})\geq 2\sqrt{n}+tn^{-1/6}\r)^2\leq \P\l(\lambda_k(H_{n,2})\geq 2\sqrt{n}+tn^{-1/6}\r)\leq \exp\l(-\frac{4k}{3}t^{3/2}(1-\vep)\r).
\end{align*}
Using Proposition \ref{prop: sum of top k stoch dom}, for $\beta\geq 2$,
\begin{align*}
     \P\l(\lambda_k(H_{n,\beta})\geq 2\sqrt{n}+tn^{-1/6}\r)&\leq \P\l(\sum\limits_{i=1}^k\lambda_i(H_{n,\beta})\geq 2k\sqrt{n}+ktn^{-1/6}\r)\\
     &\leq \P\l(\sum\limits_{i=1}^k\lambda_i(H_{\lceil\beta n/2\rceil,2})\geq 2k \sqrt{n\beta/2}+ktn^{-1/6}\sqrt{\beta/2}\r)\\
     &\leq \exp\l(-\frac{2\beta k}{3}t^{3/2}(1-\vep)\r).
\end{align*}

\subsection{Laguerre ensemble.}
We use the same approach as we did in the Hermite ensemble. We first obtain sharp upper bound for the sum of top $k$ eigenvalues in $\beta=2$ case. To this end, we use the distributional identity between the eigenvalues of LUE ($\beta=2$ case of Laguerre ensemble) and the passage times of exponential last passage percolation.

Let $\mu_{m,n,1}\geq \mu_{m,n,2}\geq ...\geq \mu_{m,n,n}$ be the random quantities such that $T_{m,n,k}:=\sum\limits_{i=1}^k\mu_{m,n,i}$ is the $k$-disjoint path last passage time in exponential environment, for all $1\leq k\leq n$. By \cite{AVMW13}[Theorem 2], $\{\mu_{m,n,i}\}_{1\leq i\leq n}$ is distributed as $\{\lambda_i\}_{1\leq i\leq n}$ in \eqref{Laguerre density} with $\beta=2.$ Here $T_{m,n,k}$ is the $k$-disjoint path last passage time $T_k((1,1)\rightarrow(m,n))$. 
Note that $T_k\left((1,1) \to (m, n)\right)$ possess the super-additive structure. That is,
\begin{align*}
    T_k({(1,1)\to (m, n)})+T_k({(m+1,n-k+1)\to (2m, 2n-k)})\leq T_k({(1,1)\to (2m, 2n)}).
\end{align*}
In the above inequality, the two summands in the left hand side are i.i.d.

We denote $\gamma:=m/n$ and $a=\l(1+\sqrt{\gamma}\r)^2$ and $b=\gamma^{-1/6}\l(1+\sqrt{\gamma}\r)^{4/3}.$ Hence, for $p\in\N$, we have,
\begin{align*}
    \P\l(T_k({(1,1)\to (m, n)})\geq kan+tbkn^{1/3}\r)^p\leq \P\l(T_k({(1,1)\to (pm, pn)})\geq akpn+tbkpn^{1/3}\r).
\end{align*}
 This gives us,
 \begin{align}\label{eq: Laguerre LDP}
\frac{\log\l(\P\l(T_k({(1,1)\to (m, n)})\geq ank+tbkn^{1/3}\r)\r)}{n}&\leq \frac{\log\l(\P\l(\frac{T_k({(1,1)\to (pm, pn)})}{pn}\geq ak+tbkn^{-2/3}\r)\r)}{pn}\notag \\
&=\frac{\log\l(\P\l(\frac{\sum\limits_{i=1}^k\lambda_i}{{pn}}\geq ak+tbkn^{-2/3}\r)\r)}{pn},
 \end{align}

where $\lambda_i$ are the top $k$ eigenvalues of LUE matrix $XX^*$ where $X$ is $pn\times pm$ matrix with i.i.d. $\mathcal{C}\mathcal{N}(0,1)$ entries. That is, the top $k$ eigenvalues of \eqref{Laguerre density} with $\beta=2$ and $n=pn$ and $m=pm$.
 We apply the large deviation result of \cite{GH22A} for the top $k$ eigenvalues of LUE.
\cite[Theorem 2.4]{GH22A} gives that the joint law of top $k$ eigenvalues of LUE satisfies large deviation principle in the scale $m+n$ with good rate function given by 
\begin{align*}
    J^{(k)}(x_1,\dots,x_k)=\frac{\beta}{2(1+\gamma)}\sum\limits_{i=1}^k\int_{(1+\sqrt{\gamma})^2}^{x_i}\frac{\sqrt{(y-(1-\sqrt{\gamma})^2))(y-(1+\sqrt{\gamma})^2))}}{y}dy.
\end{align*}

We note here that there is a typo in \cite[Theorem 2.4]{GH22A} and the factor in the denominator should be $2(1+\gamma)$ instead of $4(1+\gamma)$. 
Applying the contraction principle (see \cite{DZ98}), as sum is a continuous function, we get that $\sum\limits_{i=1}^k\lambda_i$ satisfies large deviation principle in scale $m+n$ with good rate function given by
\begin{align*}
    J_{\text{sum}}(s)=\inf\l\{\sum\limits_{i=1}^kJ(x_i):x_1\geq x_2\geq \dots\geq x_k\geq (1+\sqrt{\gamma})^2,\sum x_i=s\r\},
\end{align*}
where $J(x)$ is same as $J^{(1)}(x)$. As $J(x)$ is convex and applying Jensen's inequality gives us that $J_{\text{sum}}(s)=kJ(s/k).$ Also $J((1+\sqrt{\gamma})^2+x)/x^{3/2}\rightarrow \frac{4\gamma^{1/4}}{3(1+\sqrt{\gamma})^2(1+\gamma)}.$ Taking $p\rightarrow\infty$, for fixed $m,n$, in \eqref{eq: Laguerre LDP} we get that for any given $\vep>0$, there exists $n_{k,\vep}$ so that for any $n\geq n_{k,\vep},$ we have
\begin{align*}
\P\l(T_k({(1,1)\to (cn, n)})
\geq akn+tkbn^{1/3}\r)\leq \exp\l(-\frac{4k}{3}t^{3/2}(1-\vep)\r).
\end{align*}

This implies
\begin{align*}
    \P\l(\lambda_k(L_{n,cn,2})\geq an+tbn^{1/3}\r)\leq \P\l(\sum\limits_{i=1}^k\lambda_i(L_{n,cn,2})\geq ank+ktbn^{1/3}\r)\leq\exp\l(-\frac{4k}{3}t^{3/2}(1-\vep)\r).
\end{align*}

Using the superposition-decimation result of \eqref{eq:LOE-LSE relation}, we have that
\begin{align*}
\P\l(\lambda_k(L_{n,cn,1})\geq an+tbn^{1/3}\r)^2\leq \P\l(\lambda_k(L_{n,cn,2})\geq an+tbn^{1/3}\r)\leq \exp\l(-\frac{4k}{3}t^{3/2}(1-\vep)\r).
\end{align*}

Using Proposition \ref{prop: sum of top k stoch dom}, for $\beta\geq 2$,
\begin{align*}
\P\l(\lambda_k(L_{n,cn,\beta})\geq an+tbn^{1/3}\r)&\leq \P\l(\sum\limits_{i=1}^k\lambda_i(L_{n,cn,\beta})\geq ank+ktbn^{1/3}\r)\\
&\leq \P\l(\sum\limits_{i=1}^k\lambda_i(L_{\lceil\beta n/2\rceil,\lceil cn\beta/2\rceil,2})\geq ank\beta/2+\beta ktbn^{1/3}/2\r)\\
&\leq \exp\l(-\frac{2\beta k}{3}t^{3/2}(1-\vep)\r). 
\end{align*}
\begin{remark}\label{rem: stronger upper tail}
    We now give proof outline of Theorem \ref{thm: stronger upper tail}. It is immediate from the use of superposition-decimation result of \eqref{eq:LOE-LSE relation} and the stochastic domination result of \cite{JB23} that, \cite{BBBK25}[Theorem 1.1(ii)] can be extended to the range $\beta\geq 1$ (see the proof of \cite{BBBK25}[Theorem 1.1(i)] for details). Note that this argument works for odd $n$. For even $n$, Cauchy interlacing argument completes the proof.  For Laguerre case, the second part of \cite{BBBK25}[Theorem 1.1] holds for $m=\lfloor cn\rfloor$ for any $c>0.$ This is stronger than the version in \cite{BBBK25}[Theorem 1.1(iii)]. 
\end{remark}

\section{Stochastic domination}\label{sec: stoch dom}

The proof of Proposition \ref{prop: sum of top k stoch dom} is based on the following deterministic lemma which we will prove later. Recall that a Jacobi matrix is a tridiagonal matrix with non-negative off-diagonal entries.
\begin{lemma}\label{lem: Jacobi non-neg function}
    The sum of top $k$ eigenvalues of a symmetric Jacobi matrix is a non-decreasing function of the off-diagonal entries.
\end{lemma}
\begin{proof}[Proof of Proposition \ref{prop: sum of top k stoch dom}]
The matrices $\sqrt{\beta_2}H_{n,\beta_2}$ and $\sqrt{\beta_1}H_{s,\beta_1}$ (given in \eqref{eq: tridiag matrix}) can be coupled so that the off-diagonal entries in the top-left $n \times n$ principal sub-matrix of the latter are strictly larger than the corresponding entries in
the former, and so that the diagonal entries are equal. Using Lemma \ref{lem: Jacobi non-neg function}, it follows that    
\begin{align*}
\sqrt{\beta_2}\sum\limits_{i=1}^k\lambda_i(H_{n,\beta_2})&\overset{\text{s.t.}}{\leq} \sqrt{\beta_1}\sum\limits_{i=1}^k\lambda_i(H_{s,\beta_1}).
\end{align*}
The Laguerre case follows by similar argument, as increasing parameters results in increasing off-diagonal and diagonal entries, which along with Lemma \ref{lem: Jacobi non-neg function} gives the stochastic comparison result.
\end{proof}
\begin{proof}[Proof of Lemma \ref{lem: Jacobi non-neg function}]
    Let $A$ be a $n\times n$ Jacobi matrix. It suffices to show that the sum of top $k$ eigenvalues is a non-decreasing function of the off-diagonal entry $A_{i,i+1}=a_i$. Note that $A$ is similar to the tridiagonal matrix $B$ where $B_{i,i+1}=B_{i+1,i}=-a_i$. Indeed, take the diagonal matrix $D$ where $D_{j,j}=1$ if $j\leq i$ else it is $-1$. Taking $D^{-1}AD$, we obtain $B$, where $B_{i,i+1}=B_{i+1,i}=-a_i.$ Hence the sum of top $k$ eigenvalues are preserved if signs are flipped up to symmetry. Also the sum of top $k$ eigenvalues is a convex function of the entry $a_{i,i+1}$ (see \cite[Corollary 4.3.39]{HJ}). Hence a symmetric convex function is non-decreasing on the positive reals. As a result we have that the sum of top $k$ eigenvalues of a symmetric Jacobi matrix is a non-decreasing function of the off-diagonal entries.
\end{proof}


\begin{proof}[Proof of Theorem \ref{thm: TW tails}]
    By \eqref{eq: TW right tail bounds}
    \begin{align*}
        \P(TW_{\beta}\geq t)\leq \exp\l(-\frac{2\beta t^{3/2}}{3}(1+o(1))\r).
    \end{align*}
    After normalizing and taking limit of \eqref{eq: gen hermite comp}, we get
    \begin{align}\label{eq: TW equality}
        TW_{2/k}^{(k)}\overset{d}{=}k^{2/3}TW_{2k}^{(1)}.
    \end{align}
    
    Using \eqref{eq: TW right tail bounds} we get,
    \begin{align*}
        \P\l(TW_{2/k}^{(k)}\geq t\r)&=\P\l(k^{2/3}TW_{2k}^{(1)}\geq t\r)\\
        &\leq \exp\l(-\frac{4 t^{3/2}}{3}(1+o(1))\r).
    \end{align*}
    Taking $\alpha=(\beta'/\beta)^{2/3}$ in Theorem \ref{thm:VP result}, we get that for any $k\geq 2$ and $\beta\geq 2/k$,
    \begin{align*}
        \P\l(TW_{\beta}^{(k)}\geq t\r)\leq \exp\l(-\frac{2\beta k t^{3/2}}{3}(1+o(1))\r).
    \end{align*}
    This gives (1) in Theorem \ref{thm: TW tails}. (2) follows by taking $t$ large and fixed and then taking $n\rightarrow\infty$ limit in Theorem \ref{t: rtlb} and using \eqref{eq: Hermite TW convergence}.

Equation $1.17$ and $1.22$ of \cite{TW94} give that for any $k\in\N$,
 \begin{align*}
        \P\l(TW_{2}^{(k)}\leq -t\r)\leq \exp\l(-\frac{ t^{3}}{12}(1+o(1))\r).
    \end{align*}
 Taking $\alpha=(\beta'/\beta)^{1/3}$ in Theorem \ref{thm:VP result}, we get that for any $k\geq 2$ and $\beta\leq 2$,
    \begin{align*}
        \P\l(TW_{\beta}^{(k)}\leq -t\r)\leq \exp\l(-\frac{\beta  t^{3}}{24}(1+o(1))\r).
    \end{align*}

    \eqref{eq: TW right tail bounds} also gives that 
    \begin{align*}
         TW_{2/k}^{(ik)}\overset{d}{=}k^{2/3}TW_{2k}^{(i)}.
    \end{align*}

     Taking $\alpha=(\beta'/\beta)^{1/3}$ in Theorem \ref{thm:VP result}, we get that for any $k\geq 1$ and $\beta>0$,
    \begin{align*}
        \P\l(TW_{\beta}^{(k)}\leq -t\r)\leq \exp\l(-\frac{\beta  t^{3}}{24}(1+o(1))\r).
    \end{align*}
\end{proof}

\begin{proof}[Proof of Theorem \ref{Stoch dom s=1/3}]
    By \ref{eq: gen hermite comp}, we have
    \begin{align*}
        \lambda_{k}(H_{n_r+r,2/(r+1)})\overset{s.t.}{\geq}\lambda_{k(r+1)}(H_{n_r+r,2/(r+1)})\overset{d}{=}\sqrt{r+1}\lambda_{k}(H_{n,2(r+1)})\\
        \lambda_{k}(L_{n_r+r,m_r+r,2/(r+1)})\overset{s.t.}{\geq}\lambda_{k(r+1)}(L_{n_r+r,m_r+r,2/(r+1)})\overset{d}{=}(r+1)\lambda_{k}(L_{n,m,2(r+1)})
    \end{align*}

    Normalizing the above quantities proves the finite versions of Theorem \ref{Stoch dom s=1/3}. The second part of follows immediately Theorem \ref{Stoch dom s=1/3} from taking the $n\rightarrow\infty$ limit and using \eqref{eq: Hermite TW convergence} and \eqref{eq: Laguerre TW convergence}.
\end{proof}

\section{Left tail bounds}\label{sec: left tail bounds}
We now prove Theorem \ref{t: ltub}. 
\begin{proof}[Proof of Theorem \ref{t: ltub}]
     By \cite[Theorem 1.3]{BBBK25}, for any $\beta>0$, we have
     \[
        \P \l(\lambda_{1}(H_{n,\beta}) \leq 2\sqrt{n} -tn^{-1/6}\r) \leq \exp \l({-\frac{\beta}{24}t^{3}\l(1-\varepsilon\r)}\r).
        \]
    By \eqref{eq:GOE-GSE relation}, we also have that
    \begin{align*}
        \sqrt{2}\lambda_{1}(H_{n,4})\overset{d}{=}\lambda_{2}(H_{2n+1,1}).
    \end{align*}
    As a result 
    \begin{align*}
        \P \l(\lambda_{2}(H_{2n+1,1}) \leq 2\sqrt{2n+1} -t(2n+1)^{-1/6}\r) &\leq  \P \l(\lambda_{1}(\sqrt{2}H_{n,4}) \leq 2\sqrt{2n+1} -t(2n+1)^{-1/6}\r)\\
        &\leq \exp \l({-\frac{1}{24}t^{3}\l(1-\varepsilon\r)}\r).
    \end{align*}
   Note that this argument works for odd $n$. For even $n$, Cauchy interlacing argument completes the proof. The proof for the Laguerre ensemble is similar.
\end{proof}

\bibliography{intersection,main,main2}
 \bibliographystyle{plain}

\end{document}